\documentclass[aap,preprint]{imsart}
\RequirePackage{amsthm,amsmath,amsfonts,amssymb}
\RequirePackage[numbers]{natbib}
\RequirePackage[colorlinks,citecolor=blue,urlcolor=blue]{hyperref}
\usepackage{xcolor}

\startlocaldefs

\usepackage{graphicx}
\usepackage{mathtools,stmaryrd}
\usepackage{ulem}
\usepackage[english]{babel}
\usepackage{todonotes}
\usepackage[T1]{fontenc}
\usepackage{appendix}
\usepackage{latexsym,mathrsfs,float}
\usepackage{enumitem}

\allowdisplaybreaks

\newtheorem{theorem}{Theorem}[section]
\newtheorem{corollary}[theorem]{Corollary}
\newtheorem{lemma}[theorem]{Lemma}

\newtheorem{proposition}[theorem]{Proposition}

\newtheorem{remark}[theorem]{Remark}

\numberwithin{equation}{section}

\def\p{{\mathbb P}}
\def\e{{\mathbb E}}
\def\q{{\mathbb Q}}

\def\z{{\mathbb Z}}

\newcommand{\abs}[1]{\left\lvert #1 \right\rvert}

\endlocaldefs

\begin{document}

\begin{frontmatter}
\title{Large deviations for maximum local time of simple random walk in dimensions \texorpdfstring{$d\ge 3$}{d>=3}}
\hypersetup{pdftitle={Large Deviations for Maximum Local Time of Simple Random Walk in Dimensions d>=3}}
\runtitle{Large deviations for maximum local time}

\begin{aug}
\author[A]{\fnms{Xinyi}~\snm{Li}\ead[label=e1]{xinyili@bicmr.pku.edu.cn}}
\and
\author[B]{\fnms{Yushu}~\snm{Zheng}\ead[label=e2]{yszheng666@gmail.com}}
\address[A]{Beijing International Center for Mathematical Research, Peking University\printead[presep={,\ }]{e1}}
\address[B]{Academy of Mathematics and Systems Science, Chinese Academy of Sciences\printead[presep={,\ }]{e2}}
\end{aug}

\begin{abstract}
We obtain sharp asymptotic probabilities for upward deviations of the maximum local time of discrete- and continuous-time simple random walks on $\mathbb{Z}^d$, $d\ge 3$. For downward deviations, we prove the sharp continuous-time asymptotics and the discrete-time upper bound. Together with the loop-pruning paper~\cite{li2026loopprune_inprep}, which proves the matching discrete-time lower bound via a loop-pruning construction, this yields the sharp downward-deviation asymptotics in discrete time as well. We also derive Gumbel-type consequences at the logarithmic scale.
\end{abstract}



\end{frontmatter}
\section{Introduction}

The maximum local time of simple random walk is a classical object in the study of random walks. In dimensions $d \ge 3$, the walk is transient, so each fixed site is visited only finitely many times almost surely, while the trajectory up to time $n$ still explores many sites. The competition among these visited sites leads to an extremal quantity of logarithmic order. The classical work of Erd\H{o}s and Taylor identified this logarithmic scale, and subsequent work of R\'ev\'esz studied the maximum local time more directly in the transient setting; see \cite{ErdosTaylor1960,revesz2004maximum}. For the upper-tail mechanism relevant to the present paper, an especially useful predecessor is the work of Cs\'aki, F\"oldes, R\'ev\'esz, Rosen and Shi on frequently visited sets, whose two-point estimate for joint local times will play an important role in our analysis; see \cite{csaki2005frequently}.

A broader adjacent literature concerns thick points, frequent points, and related extreme local-time phenomena in two dimensions. On the Brownian side, Dembo, Peres, Rosen and Zeitouni developed the thick-point picture for planar Brownian motion and, as a consequence, resolved the Erd\H{o}s--Taylor conjecture for planar random walk; see \cite{dembo2001thick}. On the random-walk side, Bass and Rosen obtained asymptotics for frequent points in two dimensions; see \cite{bassrosen2007frequent}. More recently, Jego studied thick points of random walk through their connection with the Gaussian free field; see \cite{jego2020thick}. We also mention Rosen's recent tightness result for thick points in two dimensions; see \cite{rosen2023tightness}. 

While these works clarify the typical scale of the maximum local time and, in related settings, the structure of thick or frequent points, they leave open a natural large-deviation question. The present paper, together with the loop-pruning paper~\cite{li2026loopprune_inprep}, answers this question for simple random walks in transient dimensions $d\ge 3$. We remark that intimately related but different extremal problems for cover times were studied in \cite{belius2013gumbel,comets2013large,li2024large}. 

We work throughout with the discrete-time simple random walk $(S_n)_{n \ge 0}$ on $\mathbb Z^d$, $d \ge 3$ as well as the corresponding continuous-time walk obtained by attaching i.i.d.\ $\mathrm{Exp}(1)$ holding times to the discrete-time walk, which we denote by $(Y_t)_{t \ge 0}$. We write $\mathbb P$ for the law governing both walks. For $x \in \mathbb Z^d$, let
\[
\xi(n,x) := \sum_{k=0}^n \mathbf 1_{\{S_k=x\}},
\qquad
\ell(t,x) := \int_0^t \mathbf 1_{\{Y_s=x\}}\,ds
\]
be the local times of the discrete-time and continuous-time walks, respectively, and set
\[
\xi^*(n) := \max_{x \in \mathbb Z^d} \xi(n,x),
\qquad
\ell^*(t) := \max_{x \in \mathbb Z^d} \ell(t,x).
\]
Thus $\xi^*(n)$ and $\ell^*(t)$ record the local time at a most frequently visited site. Let
\[
\gamma = \gamma_d := \mathbb P\bigl(S_n \neq 0 \text{ for all } n \ge 1\bigr),
\qquad
\alpha = \alpha_d := -\frac{1}{\log(1-\gamma)}.
\]
It is classical that
\[
\lim_{t \to \infty} \frac{\ell^*(t)}{\log t} = \gamma^{-1},
\qquad
\lim_{n \to \infty} \frac{\xi^*(n)}{\log n} = \alpha
\]
almost surely; see, for instance, Theorem~13 in \cite{ErdosTaylor1960}. Thus the natural scales of the maximum local time are $\gamma^{-1}\log t$ in continuous time and $\alpha \log n$ in discrete time.

\subsection{Main results}

\begin{theorem}[Upward deviations]\label{thm:upwards}
For any $\beta > 1$ and $u \in \mathbb{R}$, as $t \to \infty$ or $n \to \infty$, we have
\begin{align}
\mathbb{P}\bigl(\ell^*(t) > \beta \cdot \gamma^{-1}\log t + u\bigr)
&= (1+o(1)) \gamma e^{-\gamma u} t^{-(\beta-1)}, \label{eq:conti_UD}\\
\mathbb{P}\bigl(\xi^*(n) > \beta \cdot \alpha \log n + u\bigr)
&= (1+o(1)) \gamma n (1-\gamma)^{\lfloor \beta \cdot \alpha \log n + u \rfloor}. \label{eq:discrete_UD}
\end{align}
\end{theorem}

Note that
\[
n^{-\beta}(1-\gamma)^u
\le (1-\gamma)^{\lfloor \beta \cdot \alpha \log n + u \rfloor}
< n^{-\beta}(1-\gamma)^{u-1}.
\]
Thus, setting
\[
c_{\beta,n,u}
:= n^\beta (1-\gamma)^{-u}(1-\gamma)^{\lfloor \beta \cdot \alpha \log n + u \rfloor}
\in [1,(1-\gamma)^{-1}),
\]
we may rewrite \eqref{eq:discrete_UD} as
\[
\mathbb{P}\bigl(\xi^*(n) > \beta \cdot \alpha \log n + u\bigr)
= (1+o(1)) c_{\beta,n,u}\gamma (1-\gamma)^u n^{-(\beta-1)}.
\]

In this paper, the downward problem is fully resolved for the continuous-time walk. For the discrete-time walk, we prove the corresponding upper bound, while the matching lower bound, which is much more involved due to the discrete-time nature of the walk, is proved in the loop-pruning paper~\cite{li2026loopprune_inprep} using a novel construction called the ``loop-pruned random walk''.

\begin{theorem}[Downward deviations]\label{thm:downward}
For any $\beta \in (0,1]$ and $u \in \mathbb{R}$, as $t \to \infty$ or $n \to \infty$, we have
\begin{align}
\mathbb{P}\bigl(\ell^*(t) \le \beta \cdot \gamma^{-1}\log t + u\bigr)
&= \exp\bigl\{-(1+o(1)) \gamma e^{-\gamma u} t^{1-\beta}\bigr\}, \label{eq:conti_DD}\\
\mathbb{P}\bigl(\xi^*(n) \le \beta \cdot \alpha \log n + u\bigr)
&\le \exp\bigl\{-(1+o(1)) c_{\beta,n,u}\gamma (1-\gamma)^u n^{1-\beta}\bigr\}. \label{eq:discrete_DD}
\end{align}
\end{theorem}
\begin{remark}
    We formulate the results for simple random walk in order to keep the exposition transparent. We expect that the arguments can be adapted to broader classes of transient finite-range random walks.
\end{remark}
\begin{remark}[Gumbel fluctuations]\label{rem:intro-gumbel}
When $\beta = 1$, the preceding results yield Gumbel-type fluctuations around the logarithmic centering:
\begin{itemize}
\item In the continuous-time case, $\ell^*(t) - \gamma^{-1}\log t$ converges in distribution to a Gumbel law with mode $\gamma^{-1}\log \gamma$ and scale parameter $\gamma^{-1}$.
\item In the discrete-time case, the present paper gives the downward upper bound. Combining this with the matching lower bound from the subsequent paper~\cite{li2026loopprune_inprep} yields the corresponding lattice, or discretized, Gumbel-type asymptotics, with the arithmetic oscillation encoded by $c_{1,n,u}$.
\end{itemize}
One may compare these results with Jego's exit-time result in dimension $d\ge3$, where for the continuous-time walk stopped at the first exit time $\tau_N$ of $V_N$, the centred maximum $\sup_{x\in V_N}\ell_{\tau_N}^x-2g\log N$ converges to a randomly shifted Gumbel law \cite[Theorem~1.2.3]{jego2020thick}. The same mesoscopic block decomposition underlying our proofs should likewise yield large-deviation estimates for the maximal local time up to exit times, by combining the deterministic-time one-block estimates with standard exit-time bounds.

One may also compare these with the Gumbel fluctuation of the cover time of tori in dimension $d \ge 3$; see~(1.2) of \cite{belius2013gumbel}.
\end{remark}
\begin{remark}[Moderate deviations]\label{rem:intro-moderate}
The estimates above can also be used in the moderate-deviation regime where the threshold approaches the typical logarithmic scale. If \(a_t\) is positive, \(a_t\to\infty\), and \(a_t=o(\log t)\), then
\begin{align*}
\mathbb P\bigl(\ell^*(t)>\gamma^{-1}\log t+a_t\bigr)
&=(1+o(1))\gamma e^{-\gamma a_t},\\
\mathbb P\bigl(\ell^*(t)\le \gamma^{-1}\log t-a_t\bigr)
&=\exp\{-(1+o(1))\gamma e^{\gamma a_t}\}.
\end{align*}
Similarly, if \(a_n\) is positive, \(a_n\to\infty\), and \(a_n=o(\log n)\), then the discrete-time upward tail satisfies
\begin{align*}
\mathbb P\bigl(\xi^*(n)>\alpha\log n+a_n\bigr)
&=(1+o(1))\gamma n(1-\gamma)^{\lfloor \alpha\log n+a_n\rfloor}.
\end{align*}
For the discrete-time downward tail, the argument in the present paper gives the corresponding upper bound
\begin{align*}
\mathbb P\bigl(\xi^*(n)\le \alpha\log n-a_n\bigr)
&\le \exp\{-(1+o(1))\gamma n(1-\gamma)^{\lfloor \alpha\log n-a_n\rfloor}\}.
\end{align*}
The matching lower bound in this moderate-deviation regime requires the discrete-time loop-pruning construction developed in the loop-pruning paper~\cite{li2026loopprune_inprep}.
\end{remark}

These results exhibit a clear asymmetry between the two tails. Upward deviations are caused by the appearance of a single unusually thick site, and their probabilities are polynomial, up to the arithmetic correction in discrete time. Downward deviations, by contrast, are collective: one must prevent the walk from creating any site whose local time exceeds a subcritical threshold, and the resulting cost is stretched exponential, with exponent $1-\beta$. 


We now briefly comment on the proof for downward deviations. The upper bounds are obtained by decomposing the trajectory into many time intervals and requiring that no interval produce a local time above the target scale. For the continuous-time lower bound, we force the walk, once a site approaches the threshold, to spend only very short future holding times there, thereby preventing any further substantial accumulation at that site. To make this strategy affordable, we further decompose the path into segments and distinguish between local accumulation, where a single segment already creates a near-thick site, and global accumulation, where several separated segments contribute to the same site. The latter mechanism is much more costly, since it requires returns across long time gaps. Under a suitable conditioning that rules out local accumulation on slightly shortened segments, one shows that, with high probability, each site is visited by only boundedly many segments, which yields control of both the number of dangerous sites and the number of their later returns.

The rest of the paper is organized as follows. Section~\ref{sec:preliminaries} collects the standard heat-kernel, hitting, and local-time estimates used throughout. Section~\ref{sec:pf_of_upwards} proves the upward-deviation asymptotics for both discrete-time and continuous-time walks. Section~\ref{sec:pf_of_downwards} treats the downward regime: it establishes the upper bounds in both settings and the continuous-time lower bound, while the discrete-time lower bound is proved in the loop-pruning paper~\cite{li2026loopprune_inprep}.

\begin{funding}
XL is supported by the National Key R\&D Program of China (No.~2021YFA1002700).
YZ is supported by the China Postdoctoral Science Foundation (Nos.~2023M743721 and 2025T180850).
\end{funding}
\section{Preliminaries}\label{sec:preliminaries}
We collect here several standard estimates for transient simple random walk on $\mathbb Z^d$, $d\ge 3$, that will be used throughout the proofs.

\subsection{Heat-kernel and hitting estimates}

We begin with the standard heat-kernel bound.

\begin{lemma}[Theorem~2.3.9 in \cite{lawler2010random}]\label{lem:lclt}
Uniformly for $n\in \mathbb N^+$, $t\in [1,\infty)$ and $x\in \mathbb Z^d$,
\begin{align}\label{eq:lclt}
\p(S_n=x)\lesssim n^{-d/2};\qquad \p(Y_t=x)\lesssim t^{-d/2}.
\end{align}
\end{lemma}

The next two simple consequences will be used repeatedly.

\begin{corollary}
Uniformly for $n\in\mathbb N^+$ and $t,r\ge 1$,
\begin{align}\label{eq:uniform_in_range}
\p(\abs{S_n}<r)\lesssim \Big(\frac{r}{n^{1/2}}\Big)^d;\qquad
\p(\abs{Y_t}<r)\lesssim \Big(\frac{r}{t^{1/2}}\Big)^d.
\end{align}
\end{corollary}

\begin{corollary}
Uniformly for $x\in \mathbb Z^d$ and $m\in \mathbb N^+$,
\begin{align}\label{eq:late_hit}
\p\big(x\in S[m,\infty)\big)\lesssim m^{-(d/2-1)}.
\end{align}
\end{corollary}

We also recall the classical asymptotic behavior of the hitting probability.

\begin{lemma}[Lemma~17.8 in \cite{revesz2013random}]\label{lem:Revesz}
There exists a constant $C_d>0$ such that
\begin{align}\label{eq:hit_prob}
\p\big(x\in S[0,\infty)\big)=\frac{C_d+o(1)}{\abs{x}^{d-2}},
\qquad \text{as }\abs{x}\to\infty.
\end{align}
\end{lemma}

\subsection{Local-time tails}

We next record the one-point and two-point local-time tails.

\begin{lemma}[Eq.~(3.1) in \cite{ErdosTaylor1960}]\label{lem:local_time_0}
For any $m\in \mathbb N$ and $s\ge 0$,
\begin{align}\label{eq:local_time_0}
\p\big(\xi(\infty,0)> m\big)=(1-\gamma)^m;
\qquad
\p\big(\ell(\infty,0)> s\big)=e^{-\gamma s}.
\end{align}
In particular, for $\beta>0$, $n\in\mathbb N$ and $t>0$,
\begin{align}\label{eq:local_time_0'}
\begin{aligned}
&\p\big(\xi(\infty,0)>\beta\alpha\log n + u\big)
=c_{\beta,n,u}(1-\gamma)^u n^{-\beta};\\ &\p\big(\ell(\infty,0)>\beta\gamma^{-1}\log t+u\big)=t^{-\beta}e^{-\gamma u}.    
\end{aligned}
\end{align}
\end{lemma}

As a key input for the upward deviation analysis, we will also need a two-point estimate.

\begin{proposition}[Eq.~(4.1) in \cite{csaki2005frequently}]\label{prop:CFRRZ}
Let
\(
t_y:=\p\big(y\in S[0,\infty)\big),\  y\in\mathbb Z^d\setminus\{0\}.
\)
Then $\xi(\infty,0)+\xi(\infty,y)$ has geometric distribution on $\mathbb N^+$ with success probability $\gamma/(1+t_y)$. Namely, for all $u\in\mathbb N$,
\begin{align}\label{eq:CFRRZ1}
\p\big(\xi(\infty,0)+\xi(\infty,y)>u\big)
=\Big(1-\frac{\gamma}{1+t_y}\Big)^u.
\end{align}
Consequently, there exists $\delta\in(0,1)$ such that for all $u\in\mathbb N$,
\begin{align}\label{eq:two_point_discrete}
\sup_{y\in \mathbb Z^d\setminus\{0\}}
\p\big(\xi(\infty,0)+\xi(\infty,y)>2u\big)
\le (1-\gamma)^{(1+\delta)u}
=\p\big(\xi(\infty,0)>u\big)^{1+\delta}.
\end{align}
\end{proposition}

\begin{corollary}\label{cor:CFRRZ}
For any $y\in \mathbb Z^d\setminus\{0\}$, the sum $\ell(\infty,0)+\ell(\infty,y)$ has exponential distribution with mean $(1+t_y)/\gamma$. Namely, for all $s\ge 0$,
\begin{align}\label{eq:two_point_continuous}
\p\big(\ell(\infty,0)+\ell(\infty,y)>s\big)
=\exp\Big(-\frac{\gamma}{1+t_y}s\Big).
\end{align}
\end{corollary}
This follows from Proposition~\ref{prop:CFRRZ}, since
$\xi(\infty,0)+\xi(\infty,y)$ is geometric with parameter $\gamma/(1+t_y)$,
and conditionally on this value the local time is a sum of i.i.d.\ $\mathrm{Exp}(1)$ variables. 

\section{Proof of Theorem \ref{thm:upwards}}\label{sec:pf_of_upwards}
\subsection{Discrete-time case}

We begin with the discrete-time simple random walk. For $x\in \mathbb{Z}^d$ and integers $0\le n_1 \le n_2$, define the local time accumulated between times $n_1$ and $n_2$ as
\[
\xi([n_1,n_2],x) := \xi(n_2,x) - \xi(n_1 - 1,x),
\]
with $\xi(-1,x):=0$.

To observe an upward deviation, it suffices for a single site to accumulate unusually large local time. This motivates a reformulation of the event in terms of a counting variable. Let
\[m_n:=\lfloor \beta\alpha\log n + u\rfloor,\quad 
\mathcal{N} = \mathcal{N}(n,\beta) := \#\big\{x\in\mathbb{Z}^d:\xi(n,x)> m_n\big\}.
\]
Since \(\xi^*(n)\) is integer-valued, \(\{\xi^*(n)>\beta\alpha\log n+u\}=\{\xi^*(n)>m_n\}\). Then we have
\[
\left\{ \xi^*(n) > m_n \right\} = \left\{ \mathcal{N} \geq 1 \right\}.
\]

To estimate $\p(\mathcal N\ge1)$, we use the first and second moment methods.
The key input is Proposition~\ref{prop:CFRRZ}, which implies that joint thick-point events at two distinct sites are negligible at the scale of the first moment. As a result, the first two moments of $\mathcal N$ are asymptotically equivalent.

\begin{lemma}\label{lem:moment}
As $n \to \infty$,
\begin{align}
	\mathbb{E}[\mathcal{N}] &= \big(1 + o(1)\big)\, c_{\beta,n,u} \gamma\,(1-\gamma)^u n^{1 - \beta}; \label{eq:first_moment}\\
	\mathbb{E}[\mathcal{N}^2] &= \big(1 + o(1)\big)\, c_{\beta,n,u} \gamma\,(1-\gamma)^u n^{1 - \beta}.\label{eq:second_moment}
\end{align}
\end{lemma}

\begin{proof}[Proof of Theorem \ref{thm:upwards} assuming Lemma \ref{lem:moment}]
This follows directly from Lemma \ref{lem:moment} via the first and second moment methods. 	
\end{proof}

We now turn to the proof of Lemma~\ref{lem:moment}. For this, we use the representation
\begin{align}\label{eq:represent}
\mathcal N=\sum_{j=0}^n \mathbf{1}\big\{\xi([j,n],S_j)=m_n+1\big\}.
\end{align}
Indeed, each site $x$ with $\xi(n,x)>m_n$ contributes exactly once, namely at the time of its $(m_n+1)$-st last visit before time $n$.

This representation reduces the moment estimates for $\mathcal N$ to estimates on the event
\(
\big\{\xi([j,n],S_j)=m_n+1\big\}.
\)
The next estimate will be used to justify replacing, away from the terminal time $n$, the event
\(
\big\{\xi([j,n],S_j)=m_n+1\big\}
\) 
by its infinite-horizon counterpart
\(
\big\{\xi([j,\infty),S_j)=m_n+1\big\},
\)
up to a negligible error.
\begin{lemma}\label{lem:accumulate_early}
For any $\beta>1$ and $\varepsilon>0$, and any integer sequence $(k_n)$ satisfying
\(
k_n=\beta\alpha\log n+O(1)
\)
as $n\to\infty$, one has
\begin{align}\label{eq:accumulate_early}
\sup_{j\ge \varepsilon n}
\p\big(\xi(\infty,0)\ge k_n,\ \xi(j,0)<k_n\big)
=o(n^{-\beta}).
\end{align}
\end{lemma}
\begin{proof}
Since $\xi(j,0)$ is nondecreasing in $j$, it suffices to prove that
\[
\p\big(\xi(\infty,0) \ge k_n,\ \xi(\lfloor \varepsilon n\rfloor,0)< k_n\big)
=o(n^{-\beta}).
\]
Let $T_0:=0$, and for $j\ge 1$ define the $j$-th return time to the origin by
\[
T_j:=\inf\{k>T_{j-1}:S_k=0\}.
\]
Then
\[
\big\{\xi(\infty,0)\ge k_n,\ \xi(\lfloor\varepsilon n\rfloor,0)<k_n\big\}
\subseteq
\{T_{k_n}<\infty,\ T_{k_n}>\varepsilon n\}.
\]

On the event $\{T_{k_n}>\varepsilon n\}$, at least one of the $k_n$ gaps
$T_j-T_{j-1}$, $1\le j\le k_n$, must be at least $L_n:=\frac{\varepsilon n}{k_n}$. Hence
\[
\{T_{k_n}<\infty,\ T_{k_n}>\varepsilon n\}
\subseteq
\bigcup_{j=1}^{k_n}\{T_{k_n}<\infty,\ T_j-T_{j-1}\ge L_n\}.
\]
Therefore,
\begin{equation}\label{eq:Aj_new}
\text{LHS of \eqref{eq:accumulate_early}}
\le
\sum_{j=1}^{k_n}\p\big(T_{k_n}<\infty,\ T_j-T_{j-1}\ge L_n\big).
\end{equation}

By the strong Markov property,
\begin{align*}
    &\p\big(T_{k_n}<\infty,\ T_j-T_{j-1}\ge L_n\big)
=
\p(T_{j-1}<\infty)\,
\p\big(T_1\ge L_n,\ T_{k_n-j}<\infty\big)\\
\le\;& \p(T_{j-1}<\infty)\,
\p(L_n\le T_1<\infty)\,
\p(T_{k_n-j}<\infty)\\
=\;&\p\big(\xi(\infty,0)>j-1\big)\,
\p\big(0\in S[\lfloor L_n\rfloor,\infty)\big)\,
\p\big(\xi(\infty,0)>k_n-j\big)
\overset{\eqref{eq:local_time_0}}{\underset{\eqref{eq:late_hit}}{\lesssim}} n^{-\beta} \Big(\frac{n}{\log n}\Big)^{-(d/2 - 1)}.
\end{align*}
Substituting this into \eqref{eq:Aj_new}, we obtain
\[
\text{LHS of \eqref{eq:accumulate_early}}
\lesssim
(\log n)\,n^{-\beta}\Big(\frac{n}{\log n}\Big)^{-(d/2-1)}
=o(n^{-\beta}).\qedhere
\]
\end{proof}
\begin{proof}[Proof of \eqref{eq:first_moment} in Lemma~\ref{lem:moment}]
By \eqref{eq:represent} and the Markov property, 
\begin{align*}
    \e[\mathcal N]
&=\sum_{j=0}^n \p\big(\xi([j,n],S_j)=m_n+1\big)\\
&=\sum_{j=0}^n \p\big(\xi(n-j,0)=m_n+1\big)
=\sum_{j=0}^n \p\big(\xi(j,0)=m_n+1\big).
\end{align*}

Fix $\varepsilon\in(0,1)$ and decompose
\[
\e[\mathcal N]
=
\sum_{0\le j<\varepsilon n}\p\big(\xi(j,0)=m_n+1\big)
+
\sum_{\varepsilon n\le j\le n}\p\big(\xi(j,0)=m_n+1\big)
=:I_1+I_2.
\]

For $I_1$, we simply use the bound
\[
I_1
\le \varepsilon n\,\p\big(\xi(\infty,0)> m_n\big)
\overset{\eqref{eq:local_time_0}}{=}\varepsilon c_{\beta,n,u}(1-\gamma)^u n^{1-\beta}.
\]

We now turn to $I_2$. Uniformly for $j\ge \varepsilon n$,
\begin{align*}
    &\big|\p\big(\xi(j,0)=m_n+1\big)-\p\big(\xi(\infty,0)=m_n+1\big)\big|\\
    \le\,& \p\big(\xi(\infty,0)> m_n,\ \xi(j,0)\le m_n\big)
    +\p\big(\xi(\infty,0)> m_n+1,\ \xi(j,0)\le m_n+1\big)
    \overset{\eqref{eq:accumulate_early}}{=}o(n^{-\beta}).
\end{align*}
Therefore,
\begin{align*}
    \p\big(\xi(j,0)=m_n+1\big)
    =\p\big(\xi(\infty,0)=m_n+1\big)+o(n^{-\beta})
    \overset{\eqref{eq:local_time_0}}{=} c_{\beta,n,u}\gamma(1-\gamma)^u n^{-\beta}+o(n^{-\beta})
\end{align*}
uniformly for $j\ge \varepsilon n$. Summing over such $j$, we obtain
\[
I_2
=
(1-\varepsilon) c_{\beta,n,u}\gamma(1-\gamma)^u n^{1-\beta}
+o(n^{1-\beta}).
\]

Combining the estimates for $I_1$ and $I_2$, we get
\begin{align*}
        (1-\varepsilon)\gamma c_{\beta,n,u}(1-\gamma)^u n^{1-\beta}&+o(n^{1-\beta})
\le
\e[\mathcal N]\\
&\le
\bigl((1-\varepsilon)\gamma+\varepsilon\bigr)c_{\beta,n,u}(1-\gamma)^un^{1-\beta}
+o(n^{1-\beta}).
\end{align*}
Since $\varepsilon>0$ is arbitrary, this proves \eqref{eq:first_moment}.
\end{proof}

\begin{proof}[Proof of \eqref{eq:second_moment} in Lemma \ref{lem:moment}]
Recall from \eqref{eq:represent} that
\[
\mathcal N=\sum_{j=0}^n X_j,
\qquad
X_j:=\mathbf 1\{\xi([j,n],S_j)=m_n+1\}.
\]
Hence
\[
\e[\mathcal N^2]-\e[\mathcal N]
=2\sum_{0\le i<j\le n}\e[X_iX_j].
\]
Moreover, if $i<j$ and $S_i=S_j$, then $X_iX_j=0$, since
\(
\xi([i,n],S_i)\ge \xi([j,n],S_j)+1.
\)
Therefore, it suffices to show that
\begin{align}\label{eq:different_ij}
\sum_{0\le i<j\le n}\e\big[X_iX_j\mathbf 1_{\{S_i\ne S_j\}}\big]
=o(n^{1-\beta}).
\end{align}

Let $\delta\in(0,1)$ be the constant from \eqref{eq:two_point_discrete}, and fix $\varepsilon$ small enough such that
\begin{align}\label{eq:eps_small}
    (1+\delta)\varepsilon \beta /2+\varepsilon<\delta\beta\qquad\text{and}\qquad \varepsilon<\beta-1.
\end{align}
We first consider the short-time regime $j-i\le n^\varepsilon$.

For $0\le i<j\le n$, define
\[
U_{i,j}:=\xi([i,j-1],S_i),
\qquad
V_{i,j}:=\xi([j,n],S_i)+\xi([j,n],S_j).
\]
If $X_i=X_j=1$ and $S_i\ne S_j$, then
\[
U_{i,j}+V_{i,j}
=\xi([i,n],S_i)+\xi([j,n],S_j)
=2(m_n+1).
\]
Set
\(
K:=\Big\lfloor\frac1\varepsilon\Big\rfloor+1.
\)
It follows that
\begin{align*}
    &\{X_i=X_j=1,\ S_i\ne S_j\}\\
\subseteq\,&
\bigcup_{k=1}^{K}
\Bigl\{
U_{i,j}\in[(k-1)\varepsilon (m_n+1),\ k\varepsilon (m_n+1)),\
V_{i,j}\ge (2-k\varepsilon)(m_n+1),\
S_i\ne S_j
\Bigr\}.
\end{align*}
Therefore,
\begin{align*}
&\e\big[X_iX_j\mathbf 1_{\{S_i\ne S_j\}}\big]\\
\le\,&
\sum_{k=1}^{K}
\p\Bigl(
U_{i,j}\in[(k-1)\varepsilon (m_n+1),\ k\varepsilon (m_n+1)),\
V_{i,j}\ge (2-k\varepsilon)(m_n+1),\
S_i\ne S_j
\Bigr)\\
\le\,& \sum_{k=1}^{K}\p\big(\xi(\infty,0)\ge (k-1)\varepsilon (m_n+1)\big)
\sup_{y\in\mathbb Z^d\setminus\{0\}}
\p\big(\xi(\infty,0)+\xi(\infty,y)\ge (2-k\varepsilon)(m_n+1)\big)\\
\le\,&
\sum_{k=1}^{K}
(1-\gamma)^{(k-1)\varepsilon (m_n+1)+(1+\delta)(1-k\varepsilon/2)(m_n+1)}\\
\le\,& K(1-\gamma)^{(1+\delta)(1-\varepsilon/2)m_n}\le C n^{-(1+\delta)\beta +(1+\delta)\varepsilon \beta /2},
\end{align*}
for some constant $C=C(\varepsilon,\beta)>0$.

Summing over all pairs $(i,j)$ with $j-i\le n^\varepsilon$, we get
\begin{align}\label{eq:sum_small}
\sum_{\substack{0\le i<j\le n:\\ j-i\le n^\varepsilon}}
\e\big[X_iX_j\mathbf 1_{\{S_i\ne S_j\}}\big]
&\le
C\,n^{1+\varepsilon}n^{-(1+\delta)\beta+(1+\delta)\varepsilon \beta /2}
\overset{\eqref{eq:eps_small}}{=}o(n^{1-\beta}).
\end{align}

Thus the short-time contribution is $o(n^{1-\beta})$. It remains to estimate the sum over pairs with $j-i>n^\varepsilon$. For this regime, we further localize the gap $j-i$. More precisely, for $\eta\in[\varepsilon,1]$, we consider pairs satisfying
\(
n^\eta\le j-i<n^{\frac{10}{9}\eta},
\)
and on each such block we distinguish between the two spatial regimes
\[
0<|S_i-S_j|<n^{\eta/8}
\qquad\text{and}\qquad
|S_i-S_j|\ge n^{\eta/8}.
\]

We begin with the near-displacement case.

\begin{lemma}\label{lem:sum_eta}
For any $\eta>0$,
\begin{align}\label{eq:sum_eta}
\sum_{\substack{0\le i<j\le n:\\ j-i\in[n^\eta,n^{\frac{10}{9}\eta})}}
\e\bigl[X_iX_j\mathbf 1_{\{0<|S_i-S_j|<n^{\eta/8}\}}\bigr]
=o(n^{1-\beta}).
\end{align}
\end{lemma}

\begin{proof}
Fix $0\le i<j\le n$ with $j-i\in[n^\eta,n^{\frac{10}{9}\eta})$. By the Markov property,
\begin{align*}
&\e\bigl[X_iX_j\mathbf 1_{\{0<|S_i-S_j|<n^{\eta/8}\}}\bigr]
\le
\e\bigl[\mathbf 1_{\{0<|S_i-S_j|<n^{\eta/8}\}}X_j\bigr]\\
\le\,&
\p\bigl(0<|S_{j-i}|<n^{\eta/8}\bigr)\,
\p\bigl(\xi(\infty,0)> m_n\bigr)\overset{\eqref{eq:uniform_in_range}}{\lesssim}\Bigl(\frac{n^{\eta/8}}{(j-i)^{1/2}}\Bigr)^d n^{-\beta}\le n^{-\frac{9}{8}\eta-\beta}.
\end{align*}
Therefore,
\[
\sum_{\substack{0\le i<j\le n:\\ j-i\in[n^\eta,n^{\frac{10}{9}\eta})}}
\e\bigl[X_iX_j\mathbf 1_{\{0<|S_i-S_j|<n^{\eta/8}\}}\bigr]
\lesssim
n^{1+\frac{10}{9}\eta}\,n^{-\frac{9}{8}\eta-\beta}
=o(n^{1-\beta}).\qedhere
\]
\end{proof}

Let
\[
K=K(\varepsilon):=\max\Bigl\{k\ge 0:\Bigl(\frac{10}{9}\Bigr)^k\varepsilon\le 1\Bigr\},
\qquad
\eta_k:=\Bigl(\frac{10}{9}\Bigr)^k\varepsilon.
\]
Applying Lemma~\ref{lem:sum_eta} on each shell, we obtain
\begin{align}\label{eq:small_dis}
\sum_{k=0}^K
\sum_{\substack{0\le i<j\le n:\\ j-i\in[n^{\eta_k},\,n^{\frac{10}{9}\eta_k})}}
\e\bigl[X_iX_j\mathbf 1_{\{0<|S_i-S_j|<n^{\eta_k/8}\}}\bigr]
=o(n^{1-\beta}).
\end{align}

We next consider the far-displacement case. By \eqref{eq:hit_prob}, uniformly for $y\in\mathbb Z^d$ with $|y|\ge n^{\eta/8}$,
\[
t_y\lesssim |y|^{-(d-2)}\le n^{-\eta/8}.
\]
Hence, by \eqref{eq:CFRRZ1}, uniformly for $u\in[0,2(m_n+1)]$ and $|y|\ge n^{\eta/8}$,
\begin{align}\label{eq:large_dis}
\p\bigl(\xi(\infty,0)+\xi(\infty,y)\ge u\bigr) =
\bigl(1-\gamma/(1+t_y)\bigr)^{u-1} =
(1-\gamma)^{u-1}\Bigl(1+O\Bigl(\frac{\log n}{n^{\eta/8}}\Bigr)\Bigr).
\end{align}

Repeating the argument from the short-time regime, but using \eqref{eq:large_dis} in place of \eqref{eq:two_point_discrete}, we obtain
\[
\e\bigl[X_iX_j\mathbf 1_{\{|S_i-S_j|\ge n^{\eta/8}\}}\bigr]
\le C n^{-(2\beta-\varepsilon)}
\]
for some constant $C=C(\varepsilon,\beta)>0$, uniformly for all $0\le i<j\le n$ satisfying
\(
j-i\in[n^\eta,n^{\frac{10}{9}\eta}).
\)
Therefore,
\begin{align}\label{eq:big_dis}
\sum_{k=0}^K
\sum_{\substack{0\le i<j\le n:\\ j-i\in[n^{\eta_k},\,n^{\frac{10}{9}\eta_k})}}
\e\bigl[X_iX_j\mathbf 1_{\{|S_i-S_j|\ge n^{\eta_k/8}\}}\bigr]
\le
(K+1)n^2\cdot C n^{-(2\beta-\varepsilon)}\overset{\eqref{eq:eps_small}}{=}
o(n^{1-\beta}).
\end{align}

Combining \eqref{eq:sum_small}, \eqref{eq:small_dis}, and \eqref{eq:big_dis}, we obtain \eqref{eq:different_ij}. This completes the proof of \eqref{eq:second_moment}.
\end{proof}

\subsection{Continuous-time case}

We now turn to the continuous-time walk. Couple the discrete-time and continuous-time simple random walks $(S_n)_{n\ge0}$ and $(Y_t)_{t\ge0}$ by
\[
Y_t=S_{N_t},
\]
where $(N_t)_{t\ge0}$ is an independent rate-$1$ Poisson process. Thus $(S_n)$ is the discrete skeleton of $(Y_t)$. For $n\in\mathbb N$, let
\[
N^{-1}(n):=\inf\{t\ge0:N_t=n\}
\]
be the $n$-th jump time of $Y$. We define the local time sampled at jump times by
\begin{align}\label{def:tilde_ell}
\widetilde{\ell}(n,x):=\ell(N^{-1}(n+1),x),
\qquad
\widetilde{\ell}^{*}(n):=\max_{x\in\mathbb Z^d}\widetilde{\ell}(n,x)=\ell^{*}(N^{-1}(n+1)).
\end{align}
For integers $1\le n_1\le n_2$, write
\[
\widetilde{\ell}([n_1,n_2],x):=\widetilde{\ell}(n_2,x)-\widetilde{\ell}(n_1-1,x).
\]

Fix $\alpha\in(1/2,1)$. Standard moderate deviations for the Poisson process imply that there exists $C=C(\alpha)>0$ such that for all $t\ge1$,
\[
\p\Bigl(\bigl|N_t-t\bigr|>t^\alpha\Bigr)\le e^{-Ct^{2\alpha-1}}.
\]
As a consequence,
\begin{align}\label{eq:moderate}
\p\Bigl(
\widetilde{\ell}^{*}(\lfloor t-t^\alpha\rfloor)
\le
\ell^{*}(t)
\le
\widetilde{\ell}^{*}(\lfloor t+t^\alpha\rfloor)
\Bigr)
\ge 1-e^{-Ct^{2\alpha-1}}.
\end{align}

Therefore, to prove the upward deviation asymptotics for $\ell^{*}(t)$, it suffices to show the following.

\begin{proposition}
Set
\(
b_n:=\beta\gamma^{-1}\log n + u
\). As $n\to\infty$,
\begin{align}\label{eq:upward_ell*}
\p\bigl(\widetilde{\ell}^{*}(n)>b_n\bigr)
=
(1+o(1))\,\gamma e^{-\gamma u}\,n^{-(\beta-1)}.
\end{align}
\end{proposition}

To this end, define
\[
\mathcal N'=\mathcal N'(n,\beta):=\#\{x\in\mathbb Z^d:\widetilde{\ell}(n,x)>b_n\}.
\]
Then $\{\widetilde{\ell}^{*}(n)>b_n\}=\{\mathcal N'\ge1\}$. Moreover, $\mathcal N'$ admits the representation
\[
\mathcal N'
=
\sum_{j=0}^n
\mathbf 1\Bigl\{
\widetilde{\ell}([j,n],S_j)>b_n,\ 
\widetilde{\ell}([j+1,n],S_j)\le b_n
\Bigr\}.
\]
Indeed, each site counted by $\mathcal N'$ contributes exactly once, namely at the last jump time at which the accumulated jump-time local time at that site crosses the level $b_n$.

The proof of \eqref{eq:upward_ell*} now proceeds in the same way as in the discrete-time case, with $\widetilde{\ell}$ in place of $\xi$, and with Corollary~\ref{cor:CFRRZ} replacing the corresponding discrete two-point estimate. Since the argument is entirely parallel, we omit the details.

\section{Proof of Theorem \ref{thm:downward}}\label{sec:pf_of_downwards}

In this section, we prove the downward deviation estimate stated in Theorem \ref{thm:downward}. Throughout, we assume that \( \beta \in (0,1] \).

\subsection{Upper bound}\label{sec:UB}
The upper bound is straightforward in both the continuous-time and discrete-time settings. We give the argument for the continuous-time walk, since the discrete-time case is completely analogous.

Fix $\beta'\in(0,\beta)$ and partition the path $Y{[0,t]}$ into sections of equal length $t^{\beta'}$:
\[Y[0,t]\supset\bigcup_{j=1}^{\lfloor t^{1-\beta'}\rfloor}Y\big[(j-1)t^{\beta'},jt^{\beta'}\big].\]
Then the event $\{\ell^*(t)\le \beta\cdot \gamma^{-1}\log t\}$ implies that the maximum local time in each section must be no more than $\beta \cdot \gamma^{-1} \log t=\frac{\beta}{\beta'}\cdot \gamma^{-1} \log t^{\beta'}$. Hence,
\begin{align*}
	&\p\big(\ell^*(t)\le \beta\cdot \gamma^{-1}\log t+u\big)
	\le\p\Big\{\ell^*\big(t^{\beta'}\big)\le \beta\cdot \gamma^{-1}\log t+u\Big\}^{\lfloor t^{1-\beta'}\rfloor}\\
	\overset{\eqref{eq:conti_UD}}{=}&\Big\{1-\big(1+o(1)\big)\gamma\cdot e^{-\gamma u}\big(t^{\beta'}\big)^{-(\beta/\beta'-1)}\Big\}^{\lfloor t^{1-\beta'}\rfloor}
	=\exp\big\{-\big(1+o(1)\big)\gamma e^{-\gamma u} t^{1-\beta}\big\}.
\end{align*}
This yields the desired upper bound.

\subsection{Lower bound}\label{sec:conti_DD}

In this subsection, we prove the lower bound in Theorem~\ref{thm:downward} for the continuous-time simple random walk. For simplicity, we restrict ourselves to the case \(u=0\), since the proof for general \(u\in\mathbb R\) is entirely analogous.

\subsubsection{Reduction to a jump-time statement and the forcing event}

We now fix \(\kappa\in(1-\beta/2,1)\). It is enough to prove the following
jump-time estimate:
\begin{align}
\mathbb{P}\Bigl(
\widetilde{\ell}^{*}\!\bigl(\lfloor n+n^\kappa\rfloor\bigr)
\le \beta \gamma^{-1}\log n
\Bigr)
\ge \exp\Bigl\{-(1+o(1))\gamma n^{1-\beta}\Bigr\},
\qquad n\to\infty .
\label{eq:conti-lower-jumptime}
\end{align}

By the moderate deviation estimate
\eqref{eq:moderate}, with probability at least \(1-\exp\{-Ct^{2\kappa-1}\}\), one has
\[
\widetilde{\ell}^{*}\!\bigl(\lfloor t-t^{\kappa}\rfloor\bigr)
\le \ell^{*}(t)
\le
\widetilde{\ell}^{*}\!\bigl(\lfloor t+t^{\kappa}\rfloor\bigr),
\]
and therefore,
\begin{align*}
&\mathbb{P}\Bigl(\ell^{*}(t)\le \beta\gamma^{-1}\log t\Bigr)\\
\ge\,&
\mathbb{P}\Bigl(
\widetilde{\ell}^{*}\!\bigl(\lfloor t+t^\kappa\rfloor\bigr)
\le \beta\gamma^{-1}\log t
\Bigr)
-
\mathbb{P}\Bigl(
\ell^{*}(t)>\widetilde{\ell}^{*}\!\bigl(\lfloor t+t^\kappa\rfloor\bigr)
\Bigr) \\
\overset{\eqref{eq:conti-lower-jumptime}}{\ge}&
\exp\Bigl\{-(1+o(1))\gamma t^{1-\beta}\Bigr\}
-
\exp\{-Ct^{2\kappa-1}\}
=
\exp\Bigl\{-(1+o(1))\gamma t^{1-\beta}\Bigr\},
\end{align*}
since \(2\kappa-1>1-\beta\). Thus \eqref{eq:conti-lower-jumptime} implies the desired
continuous-time lower bound.

It therefore suffices to prove that for every \(\delta>0\),

\begin{align}
\mathbb{P}\Bigl(
\widetilde{\ell}^{*}\!\bigl(\lfloor n+n^\kappa\rfloor\bigr)
\le \beta \gamma^{-1}\log n
\Bigr)
\ge
\exp\Bigl\{-(1+\delta)\gamma n^{1-\beta}\Bigr\}
\label{eq:conti-lower-jumptime-delta}
\end{align}
for all sufficiently large \(n\). Set \(\widehat n:=\lfloor n+n^\kappa\rfloor\).

We now describe a forcing strategy for the event
\(
\{\widetilde{\ell}^{*}(\widehat n)\le \beta\gamma^{-1}\log n\}.
\)
Fix a small \(\eta>0\), to be chosen later. 
Set
\begin{equation}
\Lambda_n:=\beta\gamma^{-1}\log n-\eta .
\label{eq:Lambda_n}
\end{equation}
The idea is to monitor those sites whose local time first reaches the preliminary level
\(\Lambda_n\), and then force all future holding times at such sites to be very short.
More precisely, once a site reaches level \(\Lambda_n\), we require the walk to leave it within
time \(\eta/2\); if the walk later returns to that site, then the next holding time is required to
be at most \(\eta/4\), then at most \(\eta/8\), and so on. Since
\(
\sum_{k=1}^{\infty}\frac{\eta}{2^{k}}=\eta,
\)
this guarantees that after the local time first reaches \(\Lambda_n\), the total additional local
time accumulated at that site is less than \(\eta\), and hence the final local time remains below
\(\beta\gamma^{-1}\log n\).

We now formalize this construction. Define stopping times \((\tau_j)_{j\ge1}\) by
\begin{equation*}
\begin{split}
    \tau_1:=\,&\inf\big\{t\ge0:\ell(t,Y_t)>\Lambda_n\big\},\qquad \mbox{and for \(j\ge2\),}
\\
\tau_j:=\,&
\inf\Bigl\{
t\ge0:\ell(t,Y_t)>\Lambda_n,\,
Y_t\notin\{Y_{\tau_1},\dots,Y_{\tau_{j-1}}\}
\Bigr\},
\end{split}
\end{equation*}
with the convention \(\inf\varnothing=\infty\). Thus \(\tau_j\) is the first time at which a new
site reaches the level \(\Lambda_n\). Let
\[
\mathcal N:=\sup\big\{j\ge1:\tau_j<N^{-1}(\widehat n)\big\},
\]
with the convention \(\sup\varnothing=0\). Hence \(\mathcal N\) is the number of distinct sites
that reach level \(\Lambda_n\) before the \(\widehat n\)-th jump time. For each \(1\le j\le \mathcal N\), define the holding time immediately after \(\tau_j\) by
\begin{equation}
h_j^{0}:=\inf\big\{s\ge0:Y_{\tau_j+s}\neq Y_{\tau_j}\big\}.
\label{eq:hj0}
\end{equation}
Recursively, for \(k\ge1\), let
\begin{equation}
\tau_j^{k}:=\inf\big\{t\ge \tau_j^{k-1}+h_j^{k-1}:Y_t=Y_{\tau_j}\big\},
\qquad
h_j^{k}:=\inf\big\{s\ge0:Y_{\tau_j^{k}+s}\neq Y_{\tau_j}\big\},
\label{eq:taujk-hjk}
\end{equation}
where \(\tau_j^{k}\) is the \(k\)-th return time to \(Y_{\tau_j}\) after \(\tau_j\), and
\(h_j^{k}\) is the corresponding holding time. We also write
\[
\mathcal M_j:=\sup\big\{k\ge1:\tau_j^{k}<N^{-1}(\widehat n)\big\},
\]
so that \(\mathcal M_j\) is the number of returns to \(Y_{\tau_j}\) before time \(N^{-1}(\widehat n)\).

We now introduce the forcing event
\begin{equation}
B:=
\Bigl\{
h_j^{k}<\eta/2^{k+1}
\ \text{for all }1\le j\le \mathcal N,\ 0\le k\le \mathcal M_j
\Bigr\}.
\label{eq:forcing_event_B}
\end{equation}

By construction, on the event \(B\), every site that ever reaches the level \(\Lambda_n\) can
accumulate at most an additional amount
\(
\sum_{k=0}^{\infty}\eta/2^{k+1}=\eta
\)
of local time afterwards. Hence no site can exceed the level
\(\Lambda_n+\eta=\beta\gamma^{-1}\log n\), and therefore
\[
B\subset
\Bigl\{
\widetilde{\ell}^{*}(\widehat n)\le \beta\gamma^{-1}\log n
\Bigr\}.
\]
Consequently,
\begin{equation}
\mathbb{P}\Bigl(
\widetilde{\ell}^{*}(\widehat n)\le \beta\gamma^{-1}\log n
\Bigr)\ge \mathbb{P}(B).
\label{eq:lower_bound_by_B}
\end{equation}
\subsubsection{Two key propositions and completion of the lower bound}

To estimate \(\mathbb P(B)\), we isolate two ingredients. The first gives the
conditional probability of the forcing event \(B\) once the number of dangerous
sites and the numbers of later returns to them are fixed. The second shows that
these quantities are typically small, with the correct exponential cost.

\begin{proposition}
\label{prop:condi_prob}
We have
\begin{align}\label{eq:condi_prob}
    \mathbb P\Bigl(
B \,\Big|\, \mathcal N,\ (\mathcal M_j)_{1\le j\le \mathcal N}
\Bigr)
=
\prod_{j=1}^{\mathcal N}\prod_{k=0}^{\mathcal M_j}
\bigl(1-e^{-\eta/2^{k+1}}\bigr).
\end{align}
\end{proposition}

\begin{proposition}
\label{prop:NM_bound}
For every \(\delta>0\), one can choose \(\eta>0\) sufficiently small so that
there exist constants \(b<1-\beta\) and \(c>0\) satisfying
\begin{align}\label{eq:NM_bound}
    \mathbb P\Bigl(
\mathcal N\le n^b,\ 
\mathcal M_j\le c\log n \ \text{for all }1\le j\le \mathcal N
\Bigr)
\ge
\exp\Bigl\{-(1+\delta/2+o(1))\gamma n^{1-\beta}\Bigr\}.
\end{align}
\end{proposition}
\begin{remark}
The exponent \(b\) in Proposition~\ref{prop:NM_bound} is allowed to be negative. In that case, \(\mathcal{N}_n\le n^b\) is equivalent to \(\mathcal{N}_n=0\).
\end{remark}
\begin{proof}[Proof of \eqref{eq:conti-lower-jumptime-delta} assuming Propositions \ref{prop:condi_prob} and \ref{prop:NM_bound}]
   From \eqref{eq:condi_prob} and \eqref{eq:NM_bound}, we have
\begin{align*}
       \p(B)&\ge \p\big(B,\,\mathcal{N}\le n^b,\,\mathcal{M}_j\le c\log n\ \forall\,1\le j\le \mathcal{N}\big)\\
       &\ge \Big[\prod_{j=1}^{\lfloor n^b \rfloor}\prod_{k=0}^{c\log n}\big(1-\exp\{-\eta/2^{k+1}\}\big)\Big]\exp\big\{-\big(1+\delta/2+o(1)\big)\,\gamma n^{1-\beta}\big\}\\
       &\ge \exp\big\{-\big(1+\delta/2+o(1)\big)\,\gamma n^{1-\beta}\big\},
   \end{align*}
which (along with \eqref{eq:lower_bound_by_B}) implies \eqref{eq:conti-lower-jumptime-delta}.
\end{proof}
\subsubsection{Proof of Proposition \ref{prop:condi_prob}}

We begin by introducing some notation.
    For $x\in \mathbb{Z}^d$ and $j\ge 1$, let $T^x_j$ denote the $j$-th hitting time of $x$:
    \begin{align*}
        T^x_1&:=\inf\big\{n\ge 0: S_n=x\big\},\qquad
        T^x_j:=\inf\big\{n > T^x_{j-1}:S_n=x\big\}\quad\text{for }j\ge 2,
    \end{align*}
    with the convention $\inf\emptyset=\infty$. Define the holding time at $x$ after the $j$-th hitting as
    \begin{align*}
        \sigma(j,x):=\mathbf{1}\{T^x_j<\infty\}\cdot\big[N^{-1}(T^x_j+1)-N^{-1}(T^x_j)\big].
    \end{align*}

Next, we focus on the points $\big(Y_{\tau_j}:1\le j\le \mathcal{N}\big)$.
For each such point, we isolate the special holding time during which the total local time reaches the threshold $\beta \cdot \gamma^{-1} \log n$.  
Specifically, for each $1 \le j \le \mathcal{N}$, define
\[\mathcal{R}_j:=\inf\Big\{k\in\mathbb{N}^+: \sum_{i=1}^k\,\sigma(i,Y_{\tau_j})\ge \Lambda_n\Big\}=\xi(\widehat n,Y_{\tau_j})-\mathcal M_j.\]
Then, immediately following the $\mathcal{R}_j$-th visit to $Y_{\tau_j}$, the addition of the holding time $\sigma(\mathcal{R}_j, Y_{\tau_j})$ causes the local time at $Y_{\tau_j}$ to exceed the threshold $\Lambda_n=\beta \gamma^{-1} \log n - \eta$.  
We split this holding time into two parts, representing the contributions before and after this threshold:
    \[\sigma(\mathcal{R}_j,Y_{\tau_j})=\sigma_1(Y_{\tau_j})+\sigma_2(Y_{\tau_j}),\] 
    with
    \begin{align*}
        \sigma_1(Y_{\tau_j}):=\Lambda_n-\sum_{i=1}^{\mathcal{R}_j-1}\sigma(i,Y_{\tau_j}),\qquad       \sigma_2(Y_{\tau_j}):=\sum_{i=1}^{\mathcal{R}_j}\sigma(i,Y_{\tau_j})-\Lambda_n.
    \end{align*}
    \begin{lemma}\label{lem:condi_hold}
        Let $(\gamma_i:i\ge 1)$ be i.i.d.\ $\mathrm{Exp}(1)$ random variables. Conditionally on $\big\{\xi(\widehat n,x):x\in\mathbb{Z}^d\big\}$, $\mathcal{N}$, $\big(Y_{\tau_j}:1\le j\le \mathcal{N}\big)$ and $\big(\mathcal{M}_j:1\le j\le \mathcal{N}\big)$, the following hold: 
        \begin{itemize}
            \item The holding times $\Big\{\sigma(i,x),\,\sigma_1(x)\mathbf{1}_{x\in \{Y_{\tau_j}\}},\, \sigma_2(x)\mathbf{1}_{x\in \{Y_{\tau_j}\}}:1\le i\le \xi(\widehat n,x)\Big\}$ are independent across different $x\in\mathbb{Z}^d$;
            \item For any $x\in \mathbb{Z}^d\setminus\{Y_{\tau_j}\}$, 
            \[\Big(\sigma(i,x):1\le i\le \xi(\widehat n,x)\Big)\overset{d}{=} \Big(\big(\gamma_i:1\le i\le \xi(\widehat n,x)\big)\,\Big|\,\sum_{i=1}^{\xi(\widehat n,x)}\gamma_i<\Lambda_n\Big);\]
            \item For any $x\in \{Y_{\tau_j}\}$,
            \begin{align*}
                &\Big(\sigma_1(x),\,\sigma(i,x):1\le i\le \mathcal{R}_j-1\Big)\overset{d}{=} \Big(\big(\gamma_i:1\le i\le \mathcal{R}_j\big)\,\Big|\,\sum_{i=1}^{\mathcal{R}_j}\gamma_i=\Lambda_n\Big),\\
                &\Big(\sigma_2(x),\sigma(i,x):\mathcal{R}_j+1\le i\le \xi(\widehat n,x)\Big)\overset{d}{=}\big(\gamma_i:1\le i\le \xi(\widehat n,x)-\mathcal{R}_j\big),
            \end{align*}
            and these two collections of holding times are mutually independent.
        \end{itemize}
    \end{lemma}
    \begin{proof}
        Conditionally on $\{\xi(\widehat n,x): x \in \mathbb{Z}^d\}$, we have:
        \begin{itemize}
            \item the holding times $\big\{\sigma(j,x):x\in \mathbb{Z}^d,\,1\le j\le \xi(\widehat n,x)\big\}$ are i.i.d. $\text{Exp}(1)$ random variables; 
            \item for any $N\in \mathbb{N}$, distinct points $(y_j : 1 \le j \le N) \subset \mathbb{Z}^d$ and non-negative integers $(m_j : 1 \le j \le N)$, the event 
            \[\Big\{\mathcal{N}=N,\,(Y_{\tau_j})=(y_j),\,(\mathcal{M}_j)=(m_j)\Big\}\] 
            equals the event that
            \begin{enumerate}[label=(\alph*), ref=\alph*]
                \item for any $x\in \mathbb{Z}^d\setminus\{y_j\}$,
                \(\sum_{i=1}^{\xi(\widehat n,x)}\sigma(i,x)<\Lambda_n,\)
                \item for any $x\in\{y_j\}$,
                \(\sum_{i=1}^{R_j-1}\sigma(i,x)\le\Lambda_n<\sum_{i=1}^{R_j}\sigma(i,x),\)
                where $R_j:=\xi(\widehat n,x)-m_j$.
            \end{enumerate}
        \end{itemize}
        The conclusions follow directly from the properties of i.i.d.\ $\text{Exp}(1)$ variables under conditioning.
    \end{proof}
\begin{proof}[Proof of Proposition \ref{prop:condi_prob}]
Recall the definitions of $h^k_j$ in \eqref{eq:hj0} and \eqref{eq:taujk-hjk}. We observe that for $1\le j\le \mathcal{N}$,
    \begin{align*}
        h^0_j=\sigma_2(Y_{\tau_j})\quad\text{and}\quad h^k_j=\sigma(\mathcal{R}_j+k,Y_{\tau_j})\quad\text{for }k\ge 1.
    \end{align*}
    Therefore, by Lemma \ref{lem:condi_hold}, conditionally on $\big\{\xi(\widehat n,x): x \in \mathbb{Z}^d\big\}$, $\mathcal{N}$, $\big(Y_{\tau_j}: 1 \le j \le \mathcal{N}\big)$, and $\big(\mathcal{M}_j : 1 \le j \le \mathcal{N}\big)$, the collection $\big(h^k_j : 1 \le j \le \mathcal{N},\ 0 \le k \le \mathcal{M}_j\big)$ consists of i.i.d.\ $\text{Exp}(1)$ random variables. This immediately yields \eqref{eq:condi_prob}.
\end{proof}
\subsection{Proof of Proposition \ref{prop:NM_bound}}

\subsubsection{Main idea}

Recall the partition of the path \( Y[0, t] \) from Section~\ref{sec:UB}. In a discrete-time analogue, we partition the path \( S[0, n] \) into segments \( S(I_j) \) (see \eqref{eq:partition} below), where each \( I_j \) represents a discrete time interval.

For any point \( x \in \mathbb{Z}^d \), the local time at \( x \) can exceed the threshold $\beta\cdot\gamma^{-1}\log n$ via one of two mechanisms:
\begin{itemize}
    \item \textbf{Local accumulation}: The local time at \( x \) accumulated during a single interval \( I_j \) already exceeds the threshold.
    \item \textbf{Global accumulation}: The local time at \( x \) remains below the threshold in each individual \( I_j \), but the cumulative sum across multiple intervals pushes it above the threshold.
\end{itemize}

Between these two, global accumulation is significantly more costly, as it requires the point \( x \) to be visited across several disjoint intervals.

As in the estimate from Section~\ref{sec:UB}, the probability that no point exceeds the threshold via local accumulation matches the lower bound. We therefore use this estimate to bound the cost of our strategy in handling the local case. On the other hand, the number of points that exceed the threshold through global accumulation is small. Consequently, the cost of controlling their subsequent holding times is also small, and can in principle be bounded by $\exp\left\{ - o(1)\, \gamma n^{1 - \beta} \right\}$.

The remaining issue is to separate the costs of ruling out local accumulation and ruling out global accumulation.
The key observation is that, after conditioning on (a suitable modification of) the event that local accumulation does not occur, the strong Markov property implies that, with high probability, each point \(x\in S[0,\widehat n]\) is visited by at most a constant number \(M=M(\beta_1,\beta_2)\) of segments \(S(I_j)\). In other words, the conditioning effectively constrains only finitely many intervals at each site.

Since the probability that no local threshold is exceeded on any fixed finite collection of intervals is \(1+o(1)\), it follows that, up to a \(1+o(1)\) factor, the event of no local accumulation and the event that our forcing strategy succeeds against global accumulation are asymptotically decoupled. This is the mechanism behind Proposition~\ref{prop:NM_bound}.

\subsubsection{Conditioning on no local accumulation and two key propositions}

We now implement this idea. Fix constants
\[
\beta_1\in(0,\beta),
\qquad
\beta_2\in\Bigl(\frac{2}{d}\beta_1,\beta_1\Bigr).
\]
We partition the path \(S[0,\widehat n]\) into segments of equal discrete length \(\lfloor n^{\beta_1}\rfloor\). Set
\(
K:=\left\lfloor \frac{\widehat n}{\lfloor n^{\beta_1}\rfloor}\right\rfloor+1.
\)
For \(j=1,\dots,K-1\), define
\[
I_j:=\bigl[(j-1)\lfloor n^{\beta_1}\rfloor,\; j\lfloor n^{\beta_1}\rfloor\bigr),
\]
and for the final interval set
\[
I_K:=\bigl[(K-1)\lfloor n^{\beta_1}\rfloor,\; \widehat n\bigr].
\]
We write the corresponding path segments as
\begin{align}\label{eq:partition}
S(I_j)=
\begin{cases}
S\bigl[(j-1)\lfloor n^{\beta_1}\rfloor,\; j\lfloor n^{\beta_1}\rfloor\bigr), & 1\le j\le K-1,\\[3pt]
S\bigl[(K-1)\lfloor n^{\beta_1}\rfloor,\; \widehat n\bigr], & j=K.
\end{cases}
\end{align}
Thus,
$
S[0,\widehat n]=\cup_{j=1}^K S(I_j)$.
Next, we shorten each segment by removing its initial \(\lfloor n^{\beta_2}\rfloor\) steps. For \(j=1,\dots,K-1\), define
\[
I_j':=\bigl[(j-1)\lfloor n^{\beta_1}\rfloor+\lfloor n^{\beta_2}\rfloor,\; j\lfloor n^{\beta_1}\rfloor\bigr),
\]
and for \(j=K\), set
\[
I_K':=\bigl[(K-1)\lfloor n^{\beta_1}\rfloor+\lfloor n^{\beta_2}\rfloor,\; \widehat n\bigr].
\]
Recall that for integers \(1\le n_1\le n_2\) and \(x\in\mathbb Z^d\),
\[
\widetilde{\ell}([n_1,n_2],x):=\widetilde{\ell}(n_2,x)-\widetilde{\ell}(n_1-1,x).
\]
We also define the maximum local time over an interval by
\[
\widetilde{\ell}^{*}([n_1,n_2])
:=\max_{x\in\mathbb Z^d}\widetilde{\ell}([n_1,n_2],x),
\qquad
\widetilde{\ell}^{*}([n_1,n_2))
:=\widetilde{\ell}^{*}([n_1,n_2-1]).
\]
To define the no-local-accumulation event, fix \(\delta>0\), and choose \(\eta>0\) sufficiently small so that
\begin{align}\label{eq:ensure}
    e^{2\gamma\eta}-1<\delta/2.
\end{align}
We then set
\begin{align}\label{def:A}
A:=\bigcap_{j=1}^K A_j,
\qquad
A_j:=\left\{
\widetilde{\ell}^{*}(I_j')
<
\beta\gamma^{-1}\log n-\eta
\right\}.
\end{align}
We first estimate the probability of a single block. Let
\[
m_n:=\lfloor n^{\beta_1}-n^{\beta_2}\rfloor .
\]
By \eqref{eq:upward_ell*}, for all sufficiently large \(n\), uniformly in \(1\le j\le K\),
\begin{align*}
\mathbb P(A_j^{\rm c})
&\le
\mathbb P\left\{
\widetilde{\ell}^{*}(m_n)
>
\frac{\beta}{\beta_1}\gamma^{-1}\log m_n-2\eta
\right\}\\
&=
(1+o(1))\,\gamma e^{2\gamma\eta}\,
(n^{\beta_1})^{-(\beta/\beta_1-1)}.
\end{align*}
In view of \eqref{eq:ensure}, this gives
\begin{align}\label{eq:single-Aj}
\mathbb P(A_j)
\ge
1-(1+\delta/2+o(1))\,\gamma\,
(n^{\beta_1})^{-(\beta/\beta_1-1)}.
\end{align}
Since the events \(A_1,\ldots,A_K\) depend on disjoint collections of increments and holding times, they are independent. Hence
\begin{align}\label{eq:pa}
\begin{aligned}
\mathbb P(A)
=\prod_{j=1}^K\mathbb P(A_j)&\ge
\left(
1-(1+\delta/2+o(1))\,\gamma\,(n^{\beta_1})^{-(\beta/\beta_1-1)}
\right)^K \\
&=
\exp\left\{-(1+\delta/2+o(1))\,\gamma\,n^{1-\beta}\right\}.
\end{aligned}
\end{align}
We denote the corresponding conditional measure by
\[
\mathbb Q=\mathbb Q(n):=\mathbb P(\,\cdot\,|\,A).
\]
It remains to prove the following two estimates under the conditional measure \(\mathbb Q\).
\begin{proposition}\label{prop:Mj_bound}
There exists a constant \(c=c(\beta_1,\beta_2)>0\) such that
\begin{align}\label{eq:o1}
\mathbb Q\bigl(\mathcal M_j\le c\log n \text{ for all }1\le j\le \mathcal N\bigr)=1-o(1).
\end{align}
\end{proposition}
\begin{proposition}\label{prop:N_bound}
There exists a constant \(b=b(\beta_1,\beta_2)<1-\beta\) such that
\begin{align}\label{eq:o2}
\mathbb Q\bigl(\mathcal N\le n^b\bigr)=1-o(1).
\end{align}
\end{proposition}
\subsubsection{Preliminaries for Propositions~\ref{prop:Mj_bound} and \ref{prop:N_bound}}
As explained in the main idea, a key step is to control how many segments can visit a given point. We begin by establishing a quantitative version of this statement. For \(1\le n_1\le n_2\le K\), write
\[
I_{[n_1,n_2]}:=I_{n_1}\cup\cdots\cup I_{n_2},
\qquad
I'_{[n_1,n_2]}:=I'_{n_1}\cup\cdots\cup I'_{n_2}.
\]
\begin{lemma}\label{lem:SlSm}
Uniformly for \(1\le i<j\le K\) with \(j-i\ge2\), \(k\in I_i\), and \(m\in I_j\), we have
\begin{align}
\mathbb Q\bigl(S_k=S_m\,\big|\,S(I_{[1,i]}),\, S(I_{[j,K]})\bigr)
&\lesssim \bigl((j-i-1)n^{\beta_2}\bigr)^{-d/2}.
\label{eq:SlSm}
\end{align}
Consequently, uniformly for \(1\le i\le K-2\) and \(k\in I_i\),
\begin{align}\label{eq:+2hit}
\mathbb Q\bigl(S_k\in S(I_{[i+2,K]})\,\big|\,S(I_{[1,i]})\bigr)
\lesssim n^{-\{(d/2)\beta_2-\beta_1\}}.
\end{align}
\end{lemma}
\begin{proof}
Fix \(k\) and \(m\). Conditionally on \(S(I_{[1,i]})\), \(S(I_{[j,K]})\), and \(S(I'_{[i+1,j-1]})\), the vector
\[
\bigl(S_{i\lfloor n^{\beta_1}\rfloor}-S_k\bigr)
+\bigl(S_m-S_{(j-1)\lfloor n^{\beta_1}\rfloor}\bigr)
+\sum_{r=i+1}^{j-1}\bigl(S_{(r+1)\lfloor n^{\beta_1}\rfloor}-S_{r\lfloor n^{\beta_1}\rfloor+\lfloor n^{\beta_2}\rfloor}\bigr)
\]
is determined by the conditioning, and hence equals some fixed vector \(x\in\mathbb Z^d\). Therefore, the event \(\{S_k=S_m\}\) is equivalent to
\[
\left\{
\sum_{r=i+1}^{j-1}\bigl(S_{r\lfloor n^{\beta_1}\rfloor+\lfloor n^{\beta_2}\rfloor}-S_{r\lfloor n^{\beta_1}\rfloor}\bigr)
=-x
\right\}.
\]
Under the conditioning, the sum above has the same law as the position of a \(d\)-dimensional simple random walk after \((j-i-1)\lfloor n^{\beta_2}\rfloor\) steps. By \eqref{eq:lclt}, there exists a universal constant \(c>0\) such that
\[
\mathbb Q\bigl(
S_k=S_m\,\big|\,S(I_{[1,i]}),\,S(I_{[j,K]}),\,S(I'_{[i+1,j-1]})
\bigr)
\le
\frac{c}{\bigl((j-i-1)n^{\beta_2}\bigr)^{d/2}}.
\]
This proves \eqref{eq:SlSm}. The bound \eqref{eq:+2hit} then follows from the union bound.
\end{proof}
For \(1\le i\le K\), let
\(
\mathcal G_i:=\sigma\bigl(S(I_{[1,i]})\bigr).
\)
The next stopping-time version follows immediately from \eqref{eq:+2hit}, by applying that estimate on each event \(\{\tau=i\}\).
\begin{corollary}\label{cor:stopping}
Let \(\tau\in[1,K-2]\) be a stopping time with respect to the filtration \((\mathcal G_i)\). Then, for some universal constant \(c>0\), conditionally on \(\mathcal G_\tau\), for any \(k\in I_\tau\),
\begin{align}\label{eq:stopping}
\mathbb Q\bigl(S_k\in S(I_{[\tau+2,K]})\,\big|\,\mathcal G_\tau\bigr)
\le c\,n^{-\{(d/2)\beta_2-\beta_1\}}.
\end{align}
\end{corollary}
We now estimate the maximal number of segments that can visit a single point. For \(x\in\mathbb Z^d\), define
\[
\mathcal J_x=\mathcal J_x(n):=\{1\le j\le K:x\in S(I_j)\}.
\]
\begin{proposition}\label{prop:Jx}
There exists a constant \(M=M(\beta_1,\beta_2)\in\mathbb N^+\) such that
\begin{align}\label{eq:Jx}
\mathbb Q\Bigl(\sup_{x\in\mathbb Z^d}\#\mathcal J_x\le M\Bigr)=1-o(1).
\end{align}
\end{proposition}
\begin{proof}
We first observe that
\begin{align}\label{eq:observe1}
\Bigl\{\sup_{x\in\mathbb Z^d}\#\mathcal J_x>M\Bigr\}
=
\bigcup_{k=0}^{\widehat n}
\bigl\{S_k\notin S[1,k-1],\ \#\mathcal J_{S_k}>M\bigr\}.
\end{align}
Fix \(1\le i\le K\) and \(k\in I_i\), and define stopping times
\[
\tau_0(k):=i,
\qquad
\tau_r(k):=\inf\{m\in[\tau_{r-1}(k)+2,K]:S_k\in S(I_m)\},
\qquad r\ge1.
\]
Then
\begin{align}\label{eq:observe2}
\bigl\{S_k\notin S[1,k-1],\ \#\mathcal J_{S_k}>M\bigr\}
\subset
\bigl\{\tau_{\lfloor M/2\rfloor-1}(k)<\infty\bigr\}.
\end{align}
We now apply \eqref{eq:stopping} iteratively to the sequence of stopping times
$
\tau_0(k),\dots,\tau_{\lfloor M/2\rfloor-2}(k)$.
This yields the following bound, for some universal constant \(c>0\):
\[
\mathbb Q\bigl(\tau_{\lfloor M/2\rfloor-1}(k)<\infty\bigr)
\le
c\,n^{-(\lfloor M/2\rfloor-1)\{(d/2)\beta_2-\beta_1\}}.
\]
Since \((d/2)\beta_2-\beta_1>0\), we may choose \(M=M(\beta_1,\beta_2)\) sufficiently large so that
\[
(\lfloor M/2\rfloor-1)\{(d/2)\beta_2-\beta_1\}>1.
\]
Combining \eqref{eq:observe1} and \eqref{eq:observe2} and using the union bound, we obtain
\[
\mathbb Q\Bigl(\sup_{x\in\mathbb Z^d}\#\mathcal J_x>M\Bigr)
\le
\widehat n\cdot c\,n^{-(\lfloor M/2\rfloor-1)\{(d/2)\beta_2-\beta_1\}}
=o(1).\qedhere
\]
\end{proof}
We record the following simple observation, which will be used in the proof of Proposition~\ref{prop:N_bound}.
For each \(1\le j\le K\), set
\[
a_j:=(j-1)\lfloor n^{\beta_1}\rfloor,\qquad
b_j:=j\lfloor n^{\beta_1}\rfloor\wedge \widehat n.
\]
Let \(\widehat Y^{(j)}\) be the \(j\)-th continuous-time block translated to start from time \(0\) and position \(0\):
\[
\widehat Y^{(j)}
:=\bigl(Y_{N^{-1}(a_j)+t}-Y_{N^{-1}(a_j)}:
0\le t\le N^{-1}(b_j)-N^{-1}(a_j)\bigr).
\]
For each subset \(J\subset\{1,\dots,K\}\), define
\begin{align}\label{def:QJ}
\mathbb Q_J:=
\mathbb P\Bigl(\,\cdot\,\Big| \bigcap_{j\in J}A_j\Bigr).
\end{align}
Thus, unlike \(\mathbb Q\), which conditions on all the events \(A_j\), the measure \(\mathbb Q_J\) conditions only on those \(A_j\) with \(j\in J\).
\begin{lemma}
For any event
\(
E\in \sigma\bigl(\widehat Y^{(j)}:j\in J\bigr),
\)
we have
\begin{align}\label{eq:QJ}
\mathbb Q(E)=\mathbb Q_J(E).
\end{align}
\end{lemma}
\subsubsection{Proof of Propositions \ref{prop:Mj_bound}}
\begin{proof}
Proposition~\ref{prop:Jx} provides a bound on the number of segments visiting any given point. We now further bound the maximal discrete local time within each segment. Together, these yield the desired bound on $\mathcal{M}_j$.
Denote by
\[
\xi^*(I_j):=\sup_{x\in\z^d}\xi(I_j,x)=\sup_{x\in\z^d}\sum_{r\in I_j}\mathbf 1_{\{S_r=x\}}
\]
the maximal discrete local time accumulated during the segment \(I_j\).
    From \eqref{eq:discrete_UD}, for each $j=1,\ldots,K$, we have
\[\p\big(\xi^*(I_j)>(1+\beta_1)\cdot\gamma^{-1}\log n\big)\lesssim (n^{\beta_1})^{-(\frac{1+\beta_1}{\beta_1}-1)}=n^{-1}.\]
	Recalling the definitions of $A$ and $A_j$ in \eqref{def:A}, it follows that
	\begin{align*}
		\q\big(\xi^*(I_j)>(1+\beta_1)\cdot\gamma^{-1}\log n\big)
        \le\p\big(\xi^*(I_j)>(1+\beta_1)\cdot\gamma^{-1}\log n\big)/\p(A_j)\lesssim n^{-1}.
	\end{align*}
	Hence, recalling that $K=\left\lfloor \frac{\widehat n}{\lfloor n^{\beta_1}\rfloor}\right\rfloor+1\sim n^{1-\beta_1}$, we obtain
	\begin{align}\label{eq:upwards}
		\q\big(\xi^*(I_j)\le (1+\beta_1)\cdot\gamma^{-1}\log n\ \forall\,j=1,\ldots,K\big)\ge 1-CKn^{-1}=1-o(1).
	\end{align}
	Now observe that
	\begin{align*}
		&\big\{\sup_{x\in\z^d}\#\mathcal{J}_x\le M\big\}\cap \big\{\xi^*(I_j)\le (1+\beta_1)\cdot\gamma^{-1}\log n\ \forall\,j=1,\ldots,K\big\}\\
        &\subset\big\{\mathcal{M}_j\le M(1+\beta_1)\cdot\gamma^{-1}\log n\ \forall\,1\le j\le \mathcal{N}\big\}.
	\end{align*}
	Combining this inclusion with the bounds from \eqref{eq:Jx} and \eqref{eq:upwards}, the desired result follows with $c=M(1+\beta_1)\cdot \gamma^{-1}$.    
\end{proof}
\subsubsection{Proof of Proposition~\ref{prop:N_bound}}
\begin{proof}
Fix \(M\in\mathbb N^+\) so that \eqref{eq:Jx} holds.
It suffices to prove that
\[
\mathbb E^{\mathbb Q}\Bigl(
\mathcal N\cdot \mathbf 1\Bigl\{\sup_{x\in\mathbb Z^d}\#\mathcal J_x\le M\Bigr\}
\Bigr)
=o(n^b)
\qquad\text{for some } b<1-\beta.
\]
\smallskip
\noindent\textbf{Step 1: Reduction to a blockwise estimate.}
We begin with the representation
\[
\mathcal N
=
\sum_{k=0}^{\widehat n}
\mathbf 1\bigl\{
S_k\notin S[0,k-1],\ \widetilde{\ell}(\widehat n,S_k)\ge \Lambda_n
\bigr\}.
\]
It follows that
\begin{align*}
\mathbb E^{\mathbb Q}\Bigl(
\mathcal N\cdot \mathbf 1\Bigl\{\sup_{x\in\mathbb Z^d}\#\mathcal J_x\le M\Bigr\}
\Bigr)\le\,&
\sum_{k=0}^{\widehat n}
\mathbb Q\Bigl(
S_k\notin S[0,k-1],\ \widetilde{\ell}(\widehat n,S_k)\ge \Lambda_n,\ \#\mathcal J_{S_k}\le M
\Bigr) \\
=\,&
\sum_{i=1}^{K}\sum_{k\in I_i}
\mathbb Q\Bigl(
S_k\notin S[0,k-1],\ \widetilde{\ell}(\widehat n,S_k)\ge \Lambda_n,\ \#\mathcal J_{S_k}\le M
\Bigr).
\end{align*}
Since \(K\sim n^{1-\beta_1}\), it remains to show that there exists \(u>0\) such that, uniformly for
\(i\in[1,K]\),
\begin{align}\label{eq:N_sum}
\sum_{k\in I_i}
\mathbb Q\Bigl(
S_k\notin S[0,k-1],\ \widetilde{\ell}(\widehat n,S_k)\ge \Lambda_n,\ \#\mathcal J_{S_k}\le M
\Bigr)
=
o\bigl(n^{-(\beta-\beta_1)-u}\bigr).
\end{align}
\smallskip
\noindent\textbf{Step 2: Decomposition according to \(\mathcal J_{S_k}\).}
To estimate the left-hand side of \eqref{eq:N_sum}, we decompose according to the possible values of
\(\mathcal J_{S_k}\):
\begin{align*}
&\mathbb Q\Bigl(
S_k\notin S[0,k-1],\ \widetilde{\ell}(\widehat n,S_k)\ge \Lambda_n,\ \#\mathcal J_{S_k}\le M
\Bigr) \\
&\qquad =
\sum_{\substack{J\subset\{i,\dots,K\}\\ \#J\le M}}
\mathbb Q\Bigl(
S_k\notin S[0,k-1],\ \widetilde{\ell}(\widehat n,S_k)\ge \Lambda_n,\ \mathcal J_{S_k}=J
\Bigr).
\end{align*}
We split the index sets \(J\) into the following three classes:
\begin{enumerate}[label=(\alph*),ref=\alph*]
    \item\label{item:case1_new} \(J=\{i\}\);
    \item\label{item:case2_new} \(J=[i,j]\) for some \(j\in[i+1,i+M-1]\);
    \item\label{item:case3_new} \(J\) is not an interval.
\end{enumerate}
We shall prove later that, uniformly for \(i\in[1,K]\), the contributions of these three classes satisfy
\begin{align}
&\sum_{k\in I_i}
\mathbb Q\Bigl(
S_k\notin S[0,k-1],\ \widetilde{\ell}(\widehat n,S_k)\ge \Lambda_n,\ \mathcal J_{S_k}=\{i\}
\Bigr)=
O\bigl(n^{-(\beta-\beta_1)-(\beta_1-\beta_2)}\bigr);\label{eq:case_a_bound}
\\[4pt]
&\sum_{j=i+1}^{(i+M-1)\wedge K}\sum_{k\in I_i}
\mathbb Q\Bigl(
S_k\notin S[0,k-1],\ \widetilde{\ell}(\widehat n,S_k)\ge \Lambda_n,\ \mathcal J_{S_k}=[i,j]
\Bigr)\label{eq:case_b_bound}\\
&\le
n^{-(\beta-\beta_1)-\beta_1/4+o(1)};
\notag\\[4pt]
&\sum_{k\in I_i}
\mathbb Q\Bigl(
S_k\notin S[0,k-1],\ \widetilde{\ell}(\widehat n,S_k)\ge \Lambda_n,\ 
\mathcal J_{S_k}\text{ is not an interval},\ \# \mathcal J_{S_k}\le M
\Bigr)\label{eq:case_c_bound}\\
&\le
n^{-(\beta-\beta_1)-((d/2)\beta_2-\beta_1)+o(1)}.
\notag
\end{align}
Since \(\beta_2<\beta_1\) and \(\beta_1<\frac d2\beta_2\), the bounds in \eqref{eq:case_a_bound}--\eqref{eq:case_c_bound} imply \eqref{eq:N_sum}.
\end{proof}
\subsubsection{Proofs of \eqref{eq:case_a_bound}--\eqref{eq:case_c_bound}}
We now prove the three case estimates \eqref{eq:case_a_bound}--\eqref{eq:case_c_bound}. Their mechanisms are as follows:
\begin{itemize}
    \item For \eqref{eq:case_a_bound}, all the local time is accumulated within a single segment. This is directly suppressed under the conditioning on \(A\), up to the short initial part removed from the segment.
    \item For \eqref{eq:case_b_bound}, the index set \(J\) is a single interval of segments. Thus, if the local time at \(S_k\) is not already created inside \(I_i\), then it must be built up by revisits coming from later segments; away from the right endpoint of \(I_i\), this forces a delayed return, which yields an additional decay.
    \item For \eqref{eq:case_c_bound}, the set \(J\) is disconnected. Each gap between successive blocks of \(J\) contributes an additional decay factor, and this gain outweighs the combinatorial complexity of summing over all such \(J\).
\end{itemize}
For later use, for each \(i\in[1,K]\), we write
\[
I_i^-:=(i-1)\lfloor n^{\beta_1}\rfloor,
\qquad
I_i^+:=i\lfloor n^{\beta_1}\rfloor\wedge \widehat n
\]
for the left and right endpoints of the interval \(I_i\).
\begin{proof}[Proof of \eqref{eq:case_a_bound}]
Fix \(i\in[1,K]\). On
\(
\Bigl\{
S_k\notin S[0,k-1],\ \widetilde{\ell}(\widehat n,S_k)\ge \Lambda_n,\ \mathcal J_{S_k}=\{i\}
\Bigr\},
\)
the site \(S_k\) is visited only during \(I_i\), and hence
\[
\widetilde{\ell}(\widehat n,S_k)=\widetilde{\ell}([k,I_i^+],S_k),
\qquad
I_i^+:=i\lfloor n^{\beta_1}\rfloor\wedge \widehat n.
\]
If \(k\in I_i'\), then \(A_i\) implies
\[
\widetilde{\ell}([k,I_i^+],S_k)\le \widetilde{\ell}^*(I_i')<\Lambda_n,
\]
so the corresponding contribution is zero. Therefore
\[
\sum_{k\in I_i}
\mathbb Q\Bigl(
S_k\notin S[0,k-1],\ \widetilde{\ell}(\widehat n,S_k)\ge \Lambda_n,\ \mathcal J_{S_k}=\{i\}
\Bigr)
\le
\sum_{k\in I_i\setminus I_i'}
\mathbb Q\Bigl(
\widetilde{\ell}([k,I_i^+],S_k)\ge \Lambda_n
\Bigr).
\]
By the definition of \(\mathbb Q\),
\[
\mathbb Q\Bigl(
\widetilde{\ell}([k,I_i^+],S_k)\ge \Lambda_n
\Bigr)
\le
\frac{\mathbb P\bigl(\widetilde{\ell}(\infty,0)\ge \Lambda_n\bigr)}{\mathbb P(A_i)}
\lesssim n^{-\beta},
\]
uniformly in \(k\in I_i\setminus I_i'\). Summing over \(k\in I_i\setminus I_i'\), and using \(\#(I_i\setminus I_i')\lesssim n^{\beta_2}\), we obtain \eqref{eq:case_a_bound}.
\end{proof}
\begin{proof}[Proof of \eqref{eq:case_b_bound}]
We first record the delayed-return estimate needed in the proof.
\begin{lemma}\label{lem:delayed_return}
Uniformly for \(m\in[1,\lfloor n^{\beta_1}\rfloor]\),
\[
\mathbb P\Big(
\widetilde{\ell}(\infty,0)\ge \Lambda_n,\ 
0\in S[m,\infty)
\Big)
\le
\begin{cases}
O(n^{-\beta}), & 1\le m\le \lfloor n^{\beta_1/2}\rfloor,\\[3pt]
n^{-(\beta+\beta_1/4)+o(1)}, & \lfloor n^{\beta_1/2}\rfloor < m\le \lfloor n^{\beta_1}\rfloor.
\end{cases}
\]
\end{lemma}
Fix \(i\in[1,K]\), \(j\in[i+1,(i+M-1)\wedge K]\), and set \(J=[i,j]\). 
Recalling the definition of \(\mathbb Q_J\) in \eqref{def:QJ}, 
\begin{align*}
&\mathbb Q\Bigl(
S_k\notin S[0,k-1],\ \widetilde{\ell}(\widehat n,S_k)\ge \Lambda_n,\ \mathcal J_{S_k}=J
\Bigr)\\
\le\,&
\mathbb Q\Bigl(
\widetilde{\ell}([k,I_j^+],S_k)\ge \Lambda_n,\ S_k\in S(I_{[i+1,j]})
\Bigr) \\
\overset{\eqref{eq:QJ}}{=}&\mathbb Q_J\Bigl(
\widetilde{\ell}([k,I_j^+],S_k)\ge \Lambda_n,\ S_k\in S(I_{[i+1,j]})
\Bigr)\\
\le\,&
\frac{
\mathbb P\Bigl(
\widetilde{\ell}([k,I_j^+],S_k)\ge \Lambda_n,\ S_k\in S(I_{[i+1,j]})
\Bigr)
}{
\prod_{\ell\in J}\mathbb P(A_\ell)
}.
\end{align*}
Since \(\#J\le M\) and \(\mathbb P(A_\ell)=1-o(1)\) uniformly in \(\ell\), the denominator is
\(1-o(1)\). Therefore, applying the Markov property and Lemma~\ref{lem:delayed_return}, we obtain
\begin{align*}
&\mathbb Q\Bigl(
S_k\notin S[0,k-1],\ \widetilde{\ell}(\widehat n,S_k)\ge \Lambda_n,\ \mathcal J_{S_k}=J
\Bigr)\\
&\qquad \le
(1+o(1))
\mathbb P\Bigl(
\widetilde{\ell}(\infty,0)\ge \Lambda_n,\ 
0\in S[\lfloor n^{\beta_1}\rfloor-k+I_i^-,\infty)
\Bigr)\\
&\qquad \le
\lfloor n^{\beta_1/2}\rfloor\cdot O(n^{-\beta})
+
\lfloor n^{\beta_1}\rfloor\cdot n^{-(\beta+\beta_1/4)+o(1)} =
n^{-(\beta-\beta_1)-\beta_1/4+o(1)}.
\end{align*}
Since there are at most $M$ admissible values of \(j\in[i+1,(i+M-1)\wedge K]\), summing over such
\(j\) does not change the order of magnitude. This proves \eqref{eq:case_b_bound}.
\end{proof}
To prove Lemma~\ref{lem:delayed_return}, we need the following auxiliary estimate.
\begin{lemma}\label{lem:general_early}
As \(n\to\infty\),
\begin{align}\label{eq:general_early}
\mathbb P\Big(
\widetilde{\ell}(\infty,0)\ge \Lambda_n,\ 
\widetilde{\ell}\big(\lfloor n^{\beta_1/2}/2\rfloor,0\big)<\Lambda_n
\Big)
\le n^{-(\beta+\beta_1/4)+o(1)}.
\end{align}
\end{lemma}
\begin{proof}
Define the inverse local time at the origin by
\[
\widetilde T_u:=\inf\{m\ge0:\widetilde{\ell}(m,0)\ge u\},
\qquad u\ge0.
\]
Fix \(\varepsilon>0\), and for \(j=1,\dots,\lfloor 1/\varepsilon\rfloor\), define
\[
\widetilde B_j:=
\Bigl\{
\varepsilon\lfloor n^{\beta_1/2}/2\rfloor
\le
\widetilde T_{j\varepsilon \Lambda_n}-\widetilde T_{(j-1)\varepsilon \Lambda_n}
<\infty
\Bigr\}.
\]
Then
\[
\Bigl\{
\widetilde{\ell}(\infty,0)\ge \Lambda_n,\ 
\widetilde{\ell}\big(\lfloor n^{\beta_1/2}/2\rfloor,0\big)<\Lambda_n
\Bigr\}
\subset
\bigcup_{j=1}^{\lfloor 1/\varepsilon\rfloor}\widetilde B_j,
\]
and therefore
\[
\mathbb P\Bigl(
\widetilde{\ell}(\infty,0)\ge \Lambda_n,\ 
\widetilde{\ell}\big(\lfloor n^{\beta_1/2}/2\rfloor,0\big)<\Lambda_n
\Bigr)
\le
\sum_{j=1}^{\lfloor 1/\varepsilon\rfloor}
\mathbb P(\widetilde T_{\Lambda_n}<\infty,\widetilde B_j).
\]
By the strong Markov property, for each \(j\),
\begin{align*}
&\mathbb P(\widetilde T_{\Lambda_n}<\infty,\widetilde B_j)\\
\le\,&
\mathbb P\bigl(\widetilde{\ell}(\infty,0)\ge (j-1)\varepsilon \Lambda_n\bigr)\,
\mathbb P\bigl(0\in S[\varepsilon\lfloor n^{\beta_1/2}/2\rfloor,\infty)\bigr)
\mathbb P\bigl(\widetilde{\ell}(\infty,0)\ge (1-j\varepsilon)\Lambda_n\bigr)\\
\le\,&
Cn^{-(1-\varepsilon)\beta-(d/2-1)\beta_1/2}
\le
Cn^{-(1-\varepsilon)\beta-\beta_1/4},
\end{align*}
since \(d\ge3\). Summing over \(j\le \lfloor 1/\varepsilon\rfloor\) and then letting
\(\varepsilon\downarrow0\), we obtain \eqref{eq:general_early}.
\end{proof}
\begin{proof}[Proof of Lemma~\ref{lem:delayed_return}]
For \(1\le m\le \lfloor n^{\beta_1/2}\rfloor\), we simply use
\[
\mathbb P\Big(
\widetilde{\ell}(\infty,0)\ge \Lambda_n,\ 
0\in S[m,\infty)
\Big)
\le
\mathbb P\bigl(\widetilde{\ell}(\infty,0)\ge \Lambda_n\bigr)
\lesssim n^{-\beta}.
\]
Now assume \(\lfloor n^{\beta_1/2}\rfloor < m\le \lfloor n^{\beta_1}\rfloor\). Splitting according to whether
\(\Lambda_n\) has been accumulated by time \(\lfloor n^{\beta_1/2}/2\rfloor\), we get
\begin{align*}
&\mathbb P\Big(
\widetilde{\ell}(\infty,0)\ge \Lambda_n,\ 
0\in S[m,\infty)
\Big)\\
\le\,&
\mathbb P\Big(
\widetilde{\ell}(\infty,0)\ge \Lambda_n,\ 
\widetilde{\ell}\big(\lfloor n^{\beta_1/2}/2\rfloor,0\big)<\Lambda_n
\Big) +
\mathbb P\Big(
\widetilde{\ell}\big(\lfloor n^{\beta_1/2}/2\rfloor,0\big)\ge \Lambda_n,\ 
0\in S[m,\infty)
\Big).
\end{align*}
The first term is bounded by \(n^{-(\beta+\beta_1/4)+o(1)}\) by Lemma~\ref{lem:general_early}. For the
second term, the Markov property gives
\begin{align*}
&\mathbb P\Big(
\widetilde{\ell}\big(\lfloor n^{\beta_1/2}/2\rfloor,0\big)\ge \Lambda_n,\ 
0\in S[m,\infty)
\Big)\\
\le\,&
\mathbb P\bigl(\widetilde{\ell}(\infty,0)\ge \Lambda_n\bigr)\,
\sup_{x\in\mathbb Z^d}
\mathbb P\Bigl(x\in S\bigl[m-\lfloor n^{\beta_1/2}/2\rfloor,\infty\bigr)\Bigr)\\
\overset{\eqref{eq:late_hit}}{=}&
n^{-(\beta+(d/2-1)\beta_1/2)+o(1)}
\le
n^{-(\beta+\beta_1/4)+o(1)}.
\end{align*}
This proves the lemma.
\end{proof}
\begin{proof}[Proof of \eqref{eq:case_c_bound}]
For \(i\in[1,K]\), let
\[
\mathbb J_i:=\bigl\{J\subset\{i,\ldots,K\}:\ \#J\le M,\ \inf J=i,\ J \text{ is not an interval}\bigr\}.
\]
We first record the uniform estimate corresponding to a fixed non-interval index set \(J\).
\begin{lemma}\label{lem:fixed_J}
Uniformly for \(i\in[1,K]\), \(J\in\mathbb J_i\), and \(k\in I_i\), the following holds.
If \(J\) admits the interval decomposition
\begin{align}\label{eq:J_decomposition}
J=[i,i_1]\cup[i_2,i_3]\cup\cdots\cup[i_{2m},i_{2m+1}],
\end{align}
where \(i=i_0\le i_1< i_2\le i_3<\cdots<i_{2m}\le i_{2m+1}\), and
\(
i_{2r}-i_{2r-1}\ge 2
\) for \(r=1,\dots,m\),
then
\begin{align}\label{eq:fixed_J}
\mathbb Q\Bigl(
\widetilde{\ell}(\widehat n,S_k)\ge \Lambda_n,\ \mathcal J_{S_k}=J
\Bigr)
\lesssim
n^{-(\beta+(d/2)\beta_2-\beta_1)+o(1)}
\prod_{r=1}^{m}\frac{1}{(i_{2r}-i_{2r-1}-1)^{d/2}}.
\end{align}
\end{lemma}
Assuming Lemma~\ref{lem:fixed_J}, we obtain
\begin{align*}
&\sum_{k\in I_i}
\mathbb Q\Bigl(
S_k\notin S[0,k-1],\ \widetilde{\ell}(\widehat n,S_k)\ge \Lambda_n,\ 
\mathcal J_{S_k}\text{ is not an interval}
\Bigr)\\
\le\;&
\sum_{J\in\mathbb J_i}\sum_{k\in I_i}
\mathbb Q\Bigl(
\widetilde{\ell}(\widehat n,S_k)\ge \Lambda_n,\ \mathcal J_{S_k}=J
\Bigr)\\
 \lesssim\;&
\lfloor n^{\beta_1}\rfloor\,
n^{-(\beta+(d/2)\beta_2-\beta_1)+o(1)}
\sum_{\substack{J\in\mathbb J_i:\\ J\text{ has decomposition }\eqref{eq:J_decomposition}}}
\prod_{r=1}^{m}\frac{1}{(i_{2r}-i_{2r-1}-1)^{d/2}}.
\end{align*}
We now need the following elementary summation estimate.
\begin{lemma}\label{lem:summation}
For every fixed \(i\in[1,K]\),
\begin{align}\label{eq:summation}
\sum_{\substack{J\in\mathbb J_i:\\ J\text{ has decomposition }\eqref{eq:J_decomposition}}}
\prod_{r=1}^{m}\frac{1}{(i_{2r}-i_{2r-1}-1)^{d/2}}
<\infty.
\end{align}
\end{lemma}
Assuming Lemma~\ref{lem:summation}, we conclude that
\begin{align*}
&\sum_{k\in I_i}
\mathbb Q\Bigl(
S_k\notin S[0,k-1],\ \widetilde{\ell}(\widehat n,S_k)\ge \Lambda_n,\ 
\mathcal J_{S_k}\text{ is not an interval}
\Bigr)\\
=\,&
n^{-(\beta-\beta_1)-((d/2)\beta_2-\beta_1)+o(1)}.
\end{align*}
This completes the proof.
\end{proof}
\begin{proof}[Proof of Lemma~\ref{lem:fixed_J}]
For an index set \(J\) with decomposition \eqref{eq:J_decomposition}, define
\[
J_r:=\bigcup_{\ell\in[i_{2r},i_{2r+1}]} I_\ell,
\qquad r=0,\dots,m,
\]
where \(i_0:=i\).
We first record the following blockwise estimate.
\begin{lemma}\label{lem:blockwise_J}
Uniformly for \(i\in[1,K]\), \(k\in I_i\), an index set \(J\) with decomposition
\eqref{eq:J_decomposition}, and \(t_0,\dots,t_m\in\mathbb R_+\), we have
\begin{equation}\label{eq:Jr}
\begin{aligned}
    &\mathbb Q\Bigl(
\widetilde{\ell}(J_r,S_k)\ge t_r\ \forall\,r=0,\ldots,m,\ \mathcal J_{S_k}=J
\Bigr)\\
\lesssim\,&
\left[
\prod_{r=1}^{m}
\frac{n^{\beta_1}}{\bigl((i_{2r}-i_{2r-1}-1)n^{\beta_2}\bigr)^{d/2}}
\right]
\exp\Bigl\{-\gamma\sum_{r=0}^{m}t_r\Bigr\}. 
\end{aligned}   
\end{equation}
\end{lemma}
Assuming Lemma~\ref{lem:blockwise_J}, we now complete the proof of
Lemma~\ref{lem:fixed_J}. Choose \(\varepsilon>0\) so small that \(\lfloor 1/\varepsilon\rfloor\ge M\). Since \(J\) contains at most \(M\) blocks, that is, \(m+1\le M\), it follows that on the event
\(
\bigl\{\widetilde{\ell}(\widehat n,S_k)\ge \Lambda_n,\ \mathcal J_{S_k}=J\bigr\},
\)
there exists \((q_0,\dots,q_m)\in\mathbb N_0^{m+1}\) such that
\[
\sum_{r=0}^{m}q_r=\lfloor 1/\varepsilon\rfloor-M\qquad\text{and}\qquad
\widetilde{\ell}(J_r,S_k)\ge q_r\varepsilon\Lambda_n,\qquad r=0,\dots,m.
\]
Hence, by the union bound,
\begin{align*}
\mathbb Q\Bigl(
\widetilde{\ell}(\widehat n,S_k)\ge \Lambda_n,\ \mathcal J_{S_k}=J
\Bigr)\le
\sum_{\substack{(q_0,\dots,q_m)\in\mathbb N_0^{m+1}:\\ \sum_{r=0}^{m}q_r=\lfloor 1/\varepsilon\rfloor-M}}
\mathbb Q\Bigl(
\widetilde{\ell}(J_r,S_k)\ge q_r\varepsilon\Lambda_n\ \forall r,\ \mathcal J_{S_k}=J
\Bigr).
\end{align*}
Applying Lemma~\ref{lem:blockwise_J}, we obtain
\begin{align*}
&\mathbb Q\Bigl(
\widetilde{\ell}(\widehat n,S_k)\ge \Lambda_n,\ \mathcal J_{S_k}=J
\Bigr)\\
\lesssim\,&
n^{-(\beta-c\varepsilon)}
\left[
\prod_{r=1}^{m}
\frac{n^{\beta_1}}{\bigl((i_{2r}-i_{2r-1}-1)n^{\beta_2}\bigr)^{d/2}}
\right]\\
&\cdot\#\Bigl\{(q_0,\dots,q_m)\in\mathbb N_0^{m+1}:\sum_{r=0}^{m}q_r=\lfloor 1/\varepsilon\rfloor-M\Bigr\}.
\end{align*} 
Since
\(
\#\Bigl\{(q_0,\dots,q_m)\in\mathbb N_0^{m+1}:\sum_{r=0}^{m}q_r=\lfloor 1/\varepsilon\rfloor-M\Bigr\}
\le (1/\varepsilon)^{M+1},
\)
letting \(\varepsilon\downarrow0\) proves \eqref{eq:fixed_J}.
\end{proof}
\begin{proof}[Proof of Lemma~\ref{lem:blockwise_J}]
Let \(J\in\mathbb J_i\) admit the same interval decomposition as in \eqref{eq:J_decomposition}.
Then for any $k\in I_i$,
\begin{align*}
&\bigl\{\widetilde{\ell}(J_r,S_k)\ge t_r \ \forall\, r\in[0,m],\ \mathcal J_{S_k}=J\bigr\} \\
&\subset
\bigcup_{(k_r)_{r=1}^m\in \prod_{r=1}^m I_{i_{2r}}}
\bigl\{
\widetilde{\ell}([k_r,I_{i_{2r+1}}^{+}],S_{k_r})\ge t_r \ \forall\, r\in[0,m],\
S_{k_r}=S_k \ \forall\, r=1,\dots,m
\bigr\},
\end{align*}
where $k_0:=k$. Consequently,
\begin{align}\label{eq:union}
\begin{aligned}
&\mathbb Q\bigl(
\widetilde{\ell}(J_r,S_k)\ge t_r \ \forall\, r\in[0,m],\ \mathcal J_{S_k}=J
\bigr) \\
\le\;&
\sum_{(k_r)_{r=1}^m\in \prod_{r=1}^m I_{i_{2r}}}
\mathbb Q\bigl(
\widetilde{\ell}([k_r,I_{i_{2r+1}}^{+}],S_{k_r})\ge t_r \ \forall\, r\in[0,m],\
S_{k_r}=S_k \ \forall\, r=1,\dots,m
\bigr).    
\end{aligned}
\end{align}
Applying \eqref{eq:SlSm} with the above \(i\) and \(k\), and iteratively for
\(j=i_{2m}, i_{2(m-1)},\ldots, i_2\), we get
\begin{align*}
    &\mathbb Q\bigl(
\widetilde{\ell}([k_r,I^+_{i_{2r+1}}],S_{k_r})\ge t_r \ \forall\, r\in[0,m],\
S_{k_r}=S_k \ \forall\, r=1,\dots,m
\bigr)\\
\le\,& \frac{1}{\big((i_{2m}-i_{2m-1}-1)n^{\beta_2}\big)^{d/2}}\\
&\cdot \mathbb Q\bigl(
\widetilde{\ell}([k_r,I^+_{i_{2r+1}}],S_{k_r})\ge t_r \ \forall\, r\in[0,m],\
S_{k_r}=S_k \ \forall\, r=1,\dots,m-1
\bigr)\\
\le\,&\cdots\le \left[
\prod_{r=1}^{m}
\frac{1}{\bigl((i_{2r}-i_{2r-1}-1)n^{\beta_2}\bigr)^{d/2}}
\right] \mathbb Q\bigl(
\widetilde{\ell}([k_r,I^+_{i_{2r+1}}],S_{k_r})\ge t_r \ \forall\, r\in[0,m]
\bigr).
\end{align*}
Moreover,
\begin{align*}
    &\mathbb Q\bigl(
\widetilde{\ell}([k_r,I^+_{i_{2r+1}}],S_{k_r})\ge t_r \ \forall\, r\in[0,m]
\bigr)\\
\overset{\eqref{eq:QJ}}{=}&\mathbb Q_J\bigl(
\widetilde{\ell}([k_r,I^+_{i_{2r+1}}],S_{k_r})\ge t_r \ \forall\, r\in[0,m]
\bigr)\\
\le\,& \big(1+o(1)\big)\mathbb P\bigl(
\widetilde{\ell}([k_r,I^+_{i_{2r+1}}],S_{k_r})\ge t_r \ \forall\, r\in[0,m]
\bigr)\\
\le\,& \prod_{r=0}^m \mathbb P\bigl(
\widetilde{\ell}(\infty,0)\ge t_r\big)=\exp\Bigl\{-\gamma\sum_{r=0}^{m}t_r\Bigr\}.
\end{align*}
Combining the above bounds, using that
\(
\#\Bigl(\prod_{r=1}^m I_{i_{2r}}\Bigr)=(n^{\beta_1})^m,
\)
and substituting into \eqref{eq:union}, we obtain \eqref{eq:Jr}.
\end{proof}
\begin{proof}[Proof of Lemma~\ref{lem:summation}]
Let \(J\in\mathbb J_i\) admit the same interval decomposition as \eqref{eq:J_decomposition}.
For each \(1\le r\le m\), define the gap length
$g_r:=i_{2r}-i_{2r-1}-1\ge1,
$
and the block lengths
\[
s_0:=i_1-i+1,
\qquad
s_r:=i_{2r+1}-i_{2r}+1\ge1,\quad 1\le r\le m.
\]
Since \(\#J\le M\), we have
\(
\sum_{r=0}^{m}s_r\le M.
\)
Thus the tuple \((m,s_0,\dots,s_m)\) can take only finitely many values.
For each fixed such tuple, summing over all possible gap sequences \((g_1,\dots,g_m)\) yields
\[
\sum_{g_1=1}^{\infty}\cdots\sum_{g_m=1}^{\infty}
\prod_{r=1}^{m}\frac{1}{g_r^{d/2}}
=
\left(\sum_{g=1}^{\infty}\frac{1}{g^{d/2}}\right)^m
<\infty,
\]
since \(d \ge 3\). This proves \eqref{eq:summation}.
\end{proof}
\bibliographystyle{imsart-number}
\bibliography{localtime}
\end{document}